\documentclass[preprint,12pt]{elsarticle}

\RequirePackage[OT1]{fontenc}
\RequirePackage[colorlinks,citecolor=blue,urlcolor=blue]{hyperref}

\usepackage{graphicx}
\usepackage{amsmath, amssymb, amsthm, mathrsfs,verbatim,bbm}

\usepackage{amsmath,amstext,amssymb,amsopn,amsthm}
\usepackage{url,verbatim}
\usepackage{mathtools}
\usepackage{enumerate}

\usepackage[toc,page]{appendix}

\usepackage{color,graphicx}

\usepackage[margin=20mm]{geometry}
\usepackage{eucal,mathrsfs,dsfont}

\allowdisplaybreaks

\newtheorem{theorem}{Theorem}[section]

\newtheorem{lemma}[theorem]{Lemma}
\newtheorem{proposition}[theorem]{Proposition}

\theoremstyle{definition}

\newtheorem{remark}[theorem]{Remark}

\numberwithin{equation}{section}

\newcommand{\eps}{\varepsilon}
\renewcommand{\epsilon}{\varepsilon}

\newcommand{\calL}{\mathcal{L}}

\newcommand{\calU}{\mathcal{U}}

\newcommand{\E}{\operatorname{\mathds{E}}} 
\renewcommand{\P}{\operatorname{\mathds{P}}} 
\newcommand{\R}{\mathds{R}}

\newcommand{\N}{{\mathds{N}}}

\newcommand{\prt}{\partial}
\newcommand{\oper}{\calU}

\newcommand{\bS}{\mathbf{S}}
\newcommand{\bu}{\mathbf{u}}
\newcommand{\be}{\mathbf{e}}
\newcommand{\bv}{\mathbf{v}}
\newcommand{\bw}{\mathbf{w}}
\newcommand{\bn}{\mathbf{n}}
\newcommand{\bb}{\mathbf{b}}
\newcommand{\bLambda}{\mathbf{\Lambda}}

\newcommand{\bU}{\mathbf{U}}
\newcommand{\bB}{\mathbf{B}}

\newcommand{\bx}{\mathbf{x}}
\newcommand{\bV}{\mathbf{V}}
\newcommand{\bX}{\mathbf{X}}

\newcommand{\wh}{\widehat}
\renewcommand{\hat}{\widehat}
\newcommand{\wt}{\widetilde}

\DeclareMathOperator{\Leb}{Leb}

\def\bv{{\bf v}}

\def\bx{{\bf x}}

\renewcommand{\v}{\sqrt{2g|y|}}
\newcommand{\vs}{2g|y|}
\newcommand{\vd}{\sqrt{2g|x_d|}}
\newcommand{\vsd}{2g|x_d|}
\newcommand{\cf}{\mathbbm{1}}

\newcommand{\cA}{\mathcal{A}}
\newcommand{\cF}{\mathcal{F}}

\journal{$ $}

\begin{document}

\begin{frontmatter}
\title{A random flight process associated to a Lorentz gas with variable density in a gravitational field}

\author[*]{Krzysztof Burdzy\corref{nsf1}}
\ead{burdzy@math.washington.edu}
\author[**]{Douglas Rizzolo\corref{nsf2}}
\ead{drizzolo@udel.edu}
\address[*]{Department of Mathematics, Box 354350,University of Washington, Seattle, WA 98195}
\address[**]{Department of Mathematical Sciences, 501 Ewing Hall, University of Delaware, Newark, DE 19716}
\cortext[nsf1]{Supported in part by NSF Grant DMS-1206276}
\cortext[nsf2]{Supported in part by NSF Grant DMS-1204840}

\begin{abstract}
We investigate the random flight process that arises as the Boltzmann-Grad limit of a random scatterer Lorentz gas with variable scatterer density in a gravitational field.  For power function densities we show how the parameters of the model determine recurrence or transience of the vertical component of the trajectory.  Finally, our methods show that, with appropriate scaling of space, time and the density of obstacles, the trajectory of the particle converges to a diffusion with explicitly given parameters. 
\end{abstract}

\end{frontmatter}

\section{Introduction}\label{intro}

We consider the random flight process that arises as the Boltzmann-Grad limit of a random scatterer model (``Lorentz gas'') in a constant gravitational field.  Lorentz gas model, which was introduced in 1905 as a model for the motion of an electron in a metallic body \cite{Lor05}, has been studied extensively in the mathematics and physics literature.  See \cite{Det14} for a recent survey.  Fundamentally, the model consists of a particle moving in an array of fixed convex scatterers, which are placed either periodically or randomly, and the particle either reflects specularly off of the scatterers (hard core model) or is pushed away via a potential (soft core model).  We are motivated by the three dimensional random scatterer hard core model where, in addition to interacting with scatterers, the particle is also pulled down by a constant gravitational field.  We generalize the process to arbitrary dimension and investigate whether it is recurrent or transient.  We show that dimension three with constant density of scatterers is critical for determining recurrence versus transience with respect to both dimension and the rate at which the density of scatterers increases.    

Various aspects of the influence of a gravitational field on a Lorentz gas have previously been investigated, see e.g. \cite{CD09, PW79, RaviT, Vy06}.  Of this prior work, only \cite{CD09} has worked directly with the Lorentz gas model.  In \cite{CD09} the authors prove the surprising result that the two dimensional periodic scatterer Lorentz gas particle in a gravitational field is recurrent and they also establish a diffusive limit for its trajectory.  However, as the authors of \cite{CD09} mention there, extending their methods to the three dimensional case currently seems intractable.   In the other papers the authors work, as we will, with the Boltzmann-Grad limit of the Lorentz gas rather than the Lorentz gas itself.  The Boltzmann-Grad limit is a low density limit in which the number of scatterers in a fixed box goes to infinity while, at the same time, the size of each scatterer goes to zero in such a way that 
the distribution of the distance between scattering events for the tracer particle has a non-degenerate limit. 
When the centers of scatterers are placed according to a Poisson process and the rates are chosen appropriately, the asymptotic behavior of the moving particle is described by a Markovian random flight process \cite{G78, S78, S88}.  The Markovian nature of the Boltzmann-Grad limit is due to the following two observations: (i) re-collisions with scatterers become unlikely as the size of each scatterer goes to zero, and (ii) the Poisson nature of the scatterer locations  means that knowing the location of one scatterer does not give information about the locations of the other scatterers.  Since analyzing the random Lorentz gas directly is beyond the capability of current techniques, this random flight model is commonly studied in both the mathematics literature \cite{BFS00, BCLM02, RaviT, Vy06} and the physics literature \cite{ADLP10, dW04, Mu04, vB05} to gain insight into the behavior of random Lorentz gas models.  Random flight processes also arise in settings other than Lorentz gas models.  For example, the random flight process we study here also appears as a model for a particle percolating through a porous medium, see \cite{WE82} and the references therein.  

\subsection{The model}
Let us now introduce our model carefully.  

We will use the notation $\bx=(x_1,x_2,\dots, x_d)\in \R^d$.
We will denote $d$-dimensional sphere $\mathbf{S}^{d-1}:=\{\bx\in \R^d: \|\bx\|=1\}$ and we will typically reserve the following notation for its elements, $\bu=(u_1,u_2,\dots, u_d) \in \bS^{d-1}$. 
We will denote components of other vectors in a similar way.
The notation $d\mathbf{x}$ will refer to $d$-dimensional Lebesgue measure.

We are primarily interested in the process in dimension three and we start by explaining the Boltzmann-Grad limit.  Fix $g>0$ and $h:\R \to \R$.  The constant $g$ will serve as the strength of the gravitational field, which will be directed towards $-\infty$ in the last coordinate and will not act on the other coordinates, and the density of scatterers will be determined by $h$.  For simplicity we will assume that the density of scatterers depends only on the distance from the plane $\R^2\times\{0\}$.  In the Boltzmann-Grad limit, we let the size of the scatterers tend to $0$ as the number of scatterers tends to $\infty$.  In particular, assume spherical scatterers with radius $1/R$ are placed so their centers are the points of a Poisson process with intensity $R^{2}h(x_3)d\mathbf{x}$.  Since, typically, the trajectory of a particle in a gravitational field does not intersect itself, the arguments of \cite{S78, S88} can easily be adapted to include the gravitational field and produce the following result: if the initial position and velocity of the particle is absolutely continuous with respect to Lebesgue measure on the constant energy surface then the distribution of the position and velocity process of the particle converges as $R\to\infty$, in the sense of convergence of finite dimensional distributions, to a Markovian random flight process $(\mathbf{X}(t), \mathbf{V}(t))_{t\geq 0}$ with generator
\begin{equation}\label{eq generator 1} \hat{D}f(\mathbf{x},\mathbf{v}) =  \mathbf{v} \cdot \nabla_{\mathbf{x}} f(\mathbf{x}, \mathbf{v}) -g\frac{\partial}{\partial v_3} f(\mathbf{x}, \mathbf{v}) + h(x_3)\|\mathbf{v}\| \int_{\mathbf{S}^{2}} (f(\mathbf{x}, \|\mathbf{v}\|\mathbf{u}) - f(\mathbf{x}, \mathbf{v})) \sigma(d\mathbf{u}),  \end{equation}
where $\sigma$ is the normalized surface measure on the unit sphere $\mathbf{S}^2$, see \cite{S88}.  More generally, in any dimension $d$ we can consider the process $(\mathbf{X}(t), \mathbf{V}(t))_{t\geq 0}$ with generator
\begin{equation}\label{eq generator d} \hat{D}f(\mathbf{x},\mathbf{v}) =  \mathbf{v} \cdot \nabla_{\mathbf{x}} f(\mathbf{x}, \mathbf{v}) -g\frac{\partial}{\partial v_d} f(\mathbf{x}, \mathbf{v}) + h(x_d)\|\mathbf{v}\| \int_{\mathbf{S}^{d-1}} (f(\mathbf{x}, \|\mathbf{v}\|\mathbf{u}) - f(\mathbf{x}, \mathbf{v})) \sigma(d\mathbf{u}),  \end{equation}
where $\sigma$ is the normalized surface measure on the unit sphere $\mathbf{S}^{d-1}$.  In dimensions other than $3$ the Boltzmann-Grad limit of the Lorentz gas has a similar generator, but instead of the integral being against the normalized surface measure it is against a kernel that depends on $\bv$, see Appendix \ref{refdir}.  We expect the two processes to have similar qualitative behavior.

 The process $(\mathbf{X}(t), \mathbf{V}(t))_{t\geq 0}$ can be constructed in the following way, which explains the name ``random flight process''.  Let $(\bLambda(\bx,\bv,t))_{t\geq 0}$ with $\bx,\bv\in \R^d$ and $\|\bv\|^2/2 -g|x_d|=E$ denote the solution to the initial value problem
 \begin{equation}\label{eq e evo} 
\left\{\begin{array}{lcc} \bLambda''  &\equiv & -g\be_d ,
\\ \bLambda(0)& = & \bx ,
\\ \bLambda'(0)& =&\ \bv,\end{array}\right.
\end{equation}
where $\be_1,\dots,\be_d$ are the standard basis vectors of $\R^d$.  We construct our process $((\mathbf{X}(t),\mathbf{V}(t)), t\geq 0 )$ recursively as follows.  Set $(\mathbf{X}(0),\mathbf{V}(0)) = (\bx,\bv)$ and let $T_0=0$. For $k\geq 1$, assuming we have defined $((\mathbf{X}(t),\mathbf{V}(t)))_{0\leq t\leq T_{k-1}}$, we let $\bU_{k-1}$ be independent of this part of the path and uniformly distributed on $\mathbf{S}^{d-1}$ and let $T_k$ satisfy
\begin{multline}\label{e_reflection} \P\left(T_k-T_{k-1}>t \mid \bU_{k-1}, ((\mathbf{X}_t,\mathbf{V}_t))_{0\leq t\leq T_{k-1}}\right)\\ =  \exp \left(- \int^{t}_0
h\left(\bLambda(\mathbf{X}(T_{k-1}),\|\mathbf{V}(T_{k-1})\|\bU_{k-1},s)\right) \left\|\bLambda'(\mathbf{X}(T_{k-1}),\|\mathbf{V}(T_{k-1})\|\bU_{k-1},s)\right\|ds\right).\end{multline}
For $t\in [T_{k-1},T_{k}]$ we then define
\begin{equation}\label{e_full_trajectory_definition}
\begin{split}
& \bX(t) := \bLambda( \bX(T_{k-1}), \|  \bV(T_{k-1})\|  \bU_{k-1}, t- T_{k-1}), \\
& \bV(t) :=\bLambda'( \bX(T_{k-1}),\|  \bV(T_{k-1})\|  \bU_{k-1}, t- T_{k-1}).
\end{split}
\end{equation}
We note that, under very mild assumptions, $T_k\to \infty$ a.s., and thus this defines the path of the particle for all times.  Intuitively, $T_k$ defines the $k$th reflection of our particle by a scatterer. 

At this point we make a simple but important observation.  By conservation of energy,
\[ \|  \bV(t)\|   = \sqrt{ 2(E +g |X_d(t)|)},\]
so that, if we define
\begin{equation}\label{eq E speed} v(\bx) =   \sqrt{2(E + g |x_d|)},\end{equation}
then
\begin{equation}\label{eq path def} \bX(t) = \bLambda\left( \bX(T_{k-1}), v(\bX(T_{k-1}))\bU_{k-1}, t- T_{k-1}\right).\end{equation}
Since $\bU_{k-1}$ is independent of  $((\mathbf{X}(t),\mathbf{V}(t)))_{0\leq t\leq T_{k-1}}$, this implies that if we define $\bX_k = \bX(T_k)$, then $(\bX_k)_{k\geq 1}$ is a Markov chain.  That the index in this chain starts at $1$ is an artifact of our deterministic choice of $\bV(0)$.  If instead of choosing $\bV(0)=\bv_0$ in the construction above we take $\bV(0) = v(\bX(0))\bU$,
with $\bU$ uniformly distributed on $\mathbf{S}^{d-1}$, then $(\bX_k)_{k\geq 0}$ is a Markov chain and its transition operator is
\begin{equation} \label{eq transition} \hat{P}f(\bx)  = \E\left[f\left(\bLambda\left(\bx,v(\bx)\bU, N(\bx,\bU)\right)\right) \right],\end{equation}
where $N(\bx,\bu)$ is a random variable with distribution
\begin{equation} \label{eq inter reflection} \P\left(N(\bx,\bu) >t\right) =  \exp \left(- \int^{t}_0 h[\bLambda(\bx,v(\bx)\bu,s)] v[\bLambda(\bx,v(\bx)\bu,s)]ds\right),\end{equation}
and conditional on $\bU=\bu$, $N(\bx,\bU)$ is distributed like $N(\bx,\bu)$.

To simplify matters, we will assume that the particle has zero total energy, i.e., $E=0$ (this is purely a normalization assumption and has no substantive impact on our results).  In this case, between reflections the particle travels along the gravitational parabola
\begin{equation} \label{equation basic parabola 1}
\left\{ \bLambda(\bx,\bu,t)  := \sum_{i=1}^{d-1} \left(x_i+u_i\sqrt{2g|x_d|}t \right)\be_i+\left( x_d+ u_d\sqrt{2g|x_d|}t - \frac{g}{2}t^2\right)\be_d , t\geq 0\right\}.
\end{equation}
  We investigate questions of transience and recurrence for the $d$'th coordinate of the random flight process \eqref{eq generator d} when $h$ is of the form $h(\mathbf{x})=h(x_d)=c|x_d|^\lambda$ for some $\lambda\geq 0$.  Since our force acts only in the $d$'th coordinate, under this assumption on $h$ the evolution of $((X_d(t),V_d(t)), t\geq 0)$ becomes a Markov process with generator
\begin{equation}\label{eq generator dth} 
Df(y,v) =  v \frac{\partial}{\partial y} f(y, v) 
-g\frac{\partial}{\partial v} f(y, v) + h(y)\sqrt{2g|y|} \int_{\mathbf{S}^{d-1}} \left(f\left(y, \sqrt{2g|y|}\mathbf{u}\right) - f(y, v)\right) \sigma(d\mathbf{u}),  
\end{equation}
and if we observe the process only at reflection times, $(X_{k,d})_{k\geq 0}$ is a Markov chain with transition operator
\begin{equation} \label{eq transition dth} 
Pf(y)  = \E\left[f\left(\Lambda_d\left(y\be_d,\sqrt{2g|y|} \bU, N(y\be_d,\bU)\right)\right) \right].
\end{equation}
For ease of notation, we set $N(y, \bu) = N(y\be_d, \bu)$.  Since our force acts only in the $d$'th coordinate, determining transience versus recurrence for the $d$'th coordinate is equivalent to determining transience versus recurrence of the particle's kinetic energy.  Our approach to transience versus recurrence naturally leads to some invariance principles, which we explore as well.  Interestingly, the scaling is non-Brownian for most values of $\lambda$.  The methods we use can also be used to establish invariance principles for more general $h$, and we sketch how this is done.  In subsequent work of the second author and other coauthors this approach was extended to study these types of random flight processes in a general force and scattering density \cite{HRW15}. 

  Our model is closely related, at least in the heuristic sense, to the Galton board dynamics considered in \cite{CD09}.  In \cite{CD09} it is shown that the trajectory of a ball in a Galton board-type billiards with gravitation is recurrent and a diffusive limit for the particle trajectory is determined.  One of the motivations of the present work is to investigate whether these results are robust under perturbations of the model.  We determine criteria for the recurrence or transience of the particle trajectory for particular forms of the density of scatterers.  Our methods allow us to derive several types of invariance principles in multiple scaling regimes and determine the influence of the density of scatterers on the limiting diffusion.  A similar model with constant scatterer density was previously considered in \cite{RaviT}, where diffusion limits were obtained but questions of transience and recurrence were not addressed. 

Suppose that
$(\mathbf{X}(t), \mathbf{V}(t))_{t\geq 0}$
  has the generator \eqref{eq generator d} and
let  $(\mathbf{X}(t))_{t\geq 0}= \{(X_1(t),\dots,X_d(t))\}_{t\geq 0}$.
The processes $(\mathbf{X}(t))_{t\geq 0}$ and $(X_d(t),t\geq 0)$ are not Markov. The concepts of recurrence and transience are typically  applied to Markov processes so we need the following definition. Let $\mathbf{0}= (0,\dots,0)$ and assume that $(\mathbf{X}(0), \mathbf{V}(0))=(\mathbf{0},\mathbf{0})$. We say that $(X_d(t),t\geq 0)$ is neighborhood recurrent if for every $y < 0$, the process $(\mathbf{X}(t), \mathbf{V}(t))_{t\geq 0}$ hits
$\R^{d-1} \times [y,0] \times \R^d$ infinitely often, a.s. 
We say that $(X_d(t),t\geq 0)$ is  recurrent if for every $y \leq 0$, the process $(\mathbf{X}(t), \mathbf{V}(t))_{t\geq 0}$ hits
$\R^{d-1} \times \{y\} \times \R^d$ infinitely often, a.s. 

Our main result on transience versus recurrence in the case $h(\mathbf{x})=h(x_d)=c|x_d|^\lambda$ is the following theorem.

\begin{theorem}\label{theorem power function intro t/r}
 Let
$(\mathbf{X}(t), \mathbf{V}(t))_{t\geq 0}$
  be the Markov process with generator \eqref{eq generator d} started from $(\mathbf{0},\mathbf{0})$ with gravitation $g$ and scatterer density $h(\mathbf{x})=h(x_d)=c|x_d|^\lambda$, with $c>0$ and $\lambda\geq 0$.  
Let  $(\mathbf{X}(t))_{t\geq 0}= \{(X_1(t),\dots,X_d(t))\}_{t\geq 0}$.
 \begin{enumerate}
 \item If $d =1$ then $(X_d(t),t\geq 0)$ is recurrent.
 \item If $d\in \{2,3\}$ then $(X_d(t) ,t\geq 0)$ is neighborhood recurrent but not recurrent.
 \item If $d\geq 4$ then $(X_d(t),t\geq 0)$ is transient if $\lambda<(d-3)/2$ and neighborhood recurrent (but not recurrent) if $\lambda>(d-3)/2$. 
\end{enumerate}
\end{theorem}

We will show that recurrence fails in the case $d\geq 2$  because $X_d(t)$ does not visit $0$ infinitely often, and 0 is the only number in $(-\infty, 0]$ with this property. 

We note that already for the case $d=3$ we have to do  careful calculations to show that the process is recurrent when $\lambda=0$, which is the ``critical'' case in dimension $3$.  In this we are aided by the fact that $h$ is constant in this case.  Even more delicate calculations are likely to be needed to determine whether the process is transient or recurrent when $d\geq 4$ and $\lambda = (d-3)/2$ so we leave this case open.

Our approach to proving Theorem \ref{theorem power function intro t/r} leads naturally to two invariance principles, the first for the process observed at reflection times and the second for the process on its natural time scale.

\begin{theorem}\label{theorem intro sk inv}
Let $(\mathbf{X}_k)_{k\geq 0}= \{(X_{1,k} ,\dots,X_{d,k})\}_{k\geq 0}$ be the Markov chain with transition operator \eqref{eq transition} with gravitation $g$ and scatterer density $h(\mathbf{x})=h(x_d)=c|x_d|^\lambda$, with $c>0$ and $\lambda\geq 0$.  Let
\[  d' =  \frac{d+1+2\lambda}{2+2\lambda}.
\]
Under these conditions, regardless of the distribution of $X_0$, 
\[ \left( \frac{1}{n^{\frac{1}{2+2\lambda}}} X_{d,[nt]},t\geq 0\right) \rightarrow_d \left(-\rho_{d'} \left(\frac{2}{dc^2}\left(1+\lambda \right)^2 t\right)^{1/(1+\lambda)}, t\geq 0\right),
\]
where the convergence is in distribution on the Skorokhod space $D(\R_+,\R)$ and $\left(\rho_{d'} (t), t\geq 0\right)$ is a $d'$-dimensional Bessel process started at $0$.
\end{theorem}

The standard classification of recurrence versus transience for Bessel processes shows that the limiting process is recurrent at $0$ if $\lambda > (d-3)/2$, transient if $\lambda < (d-3)/2$, and neighborhood recurrent at $0$ if $\lambda = (d-3)/2$.  This agrees with the classification for the process in Theorem \ref{theorem power function intro t/r}, and also predicts that the case $\lambda = (d-3)/2$ will be the most subtle.

Note that the scaling is non-Brownian except when $\lambda=0$.  Since Theorem \ref{theorem intro sk inv} deals with the process observed only at reflection times, the particle's velocity does not contribute to this exponent.  That is, the non-Brownian scaling is caused purely by the increasing scattering density.  The next result, which provides an invariance principle for $(X_d(t))_{t\geq 0}$, shows that the particle's velocity contributes a further non-Brownian term to the scaling.  Our approach uses a time change argument, but the result is somewhat weaker since the time change is degenerate when the limiting process hits $0$.  Consequently, we must stop the process before it hits $0$.  Clearly, this is only a meaningful restriction if $0$ is recurrent for the limiting process.

\begin{theorem} \label{theorem intro natural}
 Let
$(\mathbf{X}(t), \mathbf{V}(t))_{t\geq 0}$
  be the Markov process with generator \eqref{eq generator d} started from $(\mathbf{0},\mathbf{0})$ with gravitation $g$ and scatterer density $h(\mathbf{x})=h(x_d)=c|x_d|^\lambda$, with $c>0$ and $\lambda\geq 0$.  
Let  $(\mathbf{X}(t))_{t\geq 0}= \{(X_1(t),\dots,X_d(t))\}_{t\geq 0}$.
  Fix $z<v<0$.  Let $T^n_z$ be the time of the first reflection at which $X_d < n^{1/(2+2\lambda)}z$ and let $T^n_v$ be the time of the first reflection after $T^n_z$ such that $X_d > n^{1/(2+2\lambda)}v$.  Let $\mathscr{Z}$ be a diffusion on $(-\infty, 0)$ started from $z$ whose generator acts on $f\in C^2$ with compact support in $(-\infty,0)$ by
\[ \mathcal{G}^{\lambda,c} f(y) =   \frac{2\sqrt{2g}}{dc} |y|^{1/2-\lambda} \left[\frac{1}{2}f''(y) - \left(\frac{d-1-2\lambda}{4|y|} \right)f'(y) \right].\]
As $n\to \infty$ we have 
\[ \left(n^{-\frac{1}{2+2\lambda}}X_d\left(\left(n^{ \frac{3+2\lambda}{4+4\lambda}}t+T^n_z \right)\wedge  T^n_v\right),\ t \geq 0\right) \rightarrow (\mathscr{Z}(t \wedge \tau_{v+}), \ t\geq 0),\]
in distribution in the Skorokhod space $D(\R_+,\R)$, where $\tau_{v+} =\inf\{t : \mathscr{Z}(t) > v\}$.
\end{theorem}

Our methods can also be used to establish invariance principles with more general functions $h$, though with a different scaling.  We prove the following result.

\begin{theorem}\label{THM GENERAL DENSITY}
Let $h:(-\infty,0] \to \R_+$ be $C^2$ on $(-\infty,0)$ and bounded away from zero on $(-\infty, a]$ for every $a<0$. Fix $g>0$. 
 Let
$(\mathbf{X}(t), \mathbf{V}(t))_{t\geq 0}$
  be the Markov process  with generator \eqref{eq generator d} started from $((x^0_1,\dots,x^0_d),\mathbf{0})$ with $x^0_d<0$ with gravitation $g_n = g/\sqrt{n}$ and $h_n(y) = \sqrt{n}h(y)$.  Fix $x^0_d<v<0$ and define $\tau^n_{v+} = \inf\{t\geq 0 : X^n_d(t) \geq v)$. Let $\mathscr{Y}$ be a diffusion on $(-\infty,0)$ started at $y_0$, whose generator extends the operator $\cA_h$ defined below, which acts on $f\in C^2$ with compact support in $(-\infty,0)$ by
\[ \cA_hf(y) =   \frac{\v}{dh(y)} f''(y) - \frac{\v}{dh(y)} \left(\frac{d-1}{2|y|} + \frac{h'(y)}{h(y)}\right)f'(y).\]
Define $\tau_{v+} = \inf\{ t\geq 0: \mathscr{Y}_t\geq v\}$.
As $n\to \infty$, we have 
\[ \left(X_d^{n}((n^{3/4}t)\wedge  \tau^n_{v+}),\ t \geq 0\right) \to (\mathscr{Y}(t\wedge \tau_{v+}), \ t\geq 0),
\]
in distribution in the Skorokhod space $D(\R_+,\R)$.
\end{theorem}

This scaling regime is further explored in \cite{HRW15}, where more general forces and scattering densities are allowed.  The cutoff at $v$ is necessary because both the time between reflections and the distance between reflections may scale differently when the particle is near the $x$-axis.  The constant density of scatterers is a particular case of the above model and our results agree in this special case with those in \cite{RaviT}.  Although the model considered in \cite{CD09}, with large periodic obstacles, is considerably different from ours (and that in \cite{RaviT}), our results in the case of constant obstacle density agree at the heuristic level with the results in \cite{CD09}.

We note that our scaling in the invariance principle in Theorem \ref{THM GENERAL DENSITY} is anomalous in the sense that we have rescaled the spatial dynamics by a factor $\sqrt{n}$, but time must be scaled by a factor of $n^{3/4}$.  This stands in contrast to typical diffusive scaling where the spatial dynamics are rescaled by a factor of $\sqrt{n}$ and time by a factor of $n$. 

Comparing to our previous case, we see that the form of the limiting generator is the same.

\begin{proposition}
If $h(y)=c|y|^\lambda$ then $\cA_h = \mathcal{G}^{\lambda,c}$.
\end{proposition}

We note that our approach bears some similarities to other work on invariance principles related to anomalous diffusions, see e.g. \cite{MS04}, but our situation is fundamentally different.  In the current setting the particle's speed is unbounded so that the waiting time between reflections can be very small and this contributes to the anomalous scaling.  However, although the scaling is anomalous, our limiting diffusion is not.  This is in contrast to \cite{MS04} and other work on anomalous diffusion where the anomalous scaling arises because waiting times can be heavy tailed.  Since Theorem \ref{THM GENERAL DENSITY} is a result about approximation of a one-dimensional diffusion there are other approaches as well, for example using \cite{MR0365224}.  The general literature on billiards, billiards with potential, and on Lorentz gas models is huge and we do not feel that we can do justice to this body of research. The articles \cite{CD09, RaviT} and references therein are a good point of entry to this field.   

This article is organized as follows. In Section \ref{se method} we consider a simplified model where the particle travels distance exactly one between reflections.  The computations in this case are simpler and the model illustrates the approach we take in the general case.  Section \ref{highdimpower} is devoted to the proofs of Theorems  \ref{theorem power function intro t/r}, \ref{theorem intro sk inv}, and \ref{theorem intro natural}, with Section \ref{se basic estimates} containing technical estimates and Section \ref{se proofs} containing the proofs of the theorems.  

\section{An overview of the method}\label{se method}

Our approach is to employ results developed by Lamperti \cite{L60,L62,L63}.  These papers provide a general framework for establishing recurrence or transience of nonnegative Markov processes. We collect and combine several results of Lamperti in Theorem \ref{theorem Lamperti} below.

Given $A\geq 0$, we will say that a non-negative stochastic process $(X_m,m\geq 0)$ is $A$-recurrent if $\P(X_m \in [0,A] \ \text{i.o.}) =1$.

\begin{theorem}\label{theorem Lamperti}
Let $(X_m,m\geq 0)$ be a Markov chain on $[0,\infty)$ with transition operator $\mathcal T$ and for $\vartheta\in \R$, let 
\[ \mu^\vartheta_k(x) = \E\left[ \left(X^{(2-\vartheta)/2}_{n+1}-X^{(2-\vartheta)/2}_n\right)^k \ \middle| \  X_n=x\right].\]
When $\vartheta=0$ we suppress it in the notation.  That is, we set $\mu_k = \mu_k^0$.
Assume:
\begin{enumerate}
\item There exists $\vartheta <2$ such that, as $x\to \infty$, $x^{1-\vartheta} \mu_1(x) \to a$, $ x^{-\vartheta}\mu_2(x)\to b>0$ with $2a+ b(1-\vartheta)>0$ and for each fixed $k\in \N$, $\mu_k(x) = O(x^{k\vartheta/2})$.  
\item $\mathcal T$ maps the set $C_0(\R_+,\R)$ of continuous functions from $[0,\infty) \to \R$ that vanish at $\infty$ to itself. 
\item\label{4} $\P(\limsup X_n=\infty \mid X_0=x)=1$ for all $x\in [0,\infty)$. 
\end{enumerate}
Let
\[ c= \frac{b\left(1-\vartheta \right) + 2a}{b\left(1-\frac{\vartheta}{2}\right)}.
\]

(a) Regardless of the distribution of $X_0$, 
\[ \left( \frac{1}{n^{\frac{1}{2-\vartheta}}} X_{[nt]},t\geq 0\right) \rightarrow_d \left(\rho_{c} \left(b\left(1-\frac{\vartheta}{2}\right)^2 t\right)^{2/(2-\vartheta)}, t\geq 0\right),
\]
where the convergence is in distribution on the Skorokhod space $D(\R_+,\R)$ and $\left(\rho_{c} (t), t\geq 0\right)$ is a $c$-dimensional Bessel process started at $0$.  

(b) If $2a>b$ then the process is transient.

(c) If $2a<b$ then there exists $A\geq 0$ such that the process is $A$-recurrent.

(d) If, for $\vartheta$ as in Assumption 1, $2x\mu^\vartheta_1(x) - \mu^\vartheta_2(x) = O(x^{-\epsilon})$ for some $\epsilon>0$ then there exists $A\geq 0$ such that the process is $A$-recurrent.
\end{theorem}

\begin{remark}
What we are calling $A$-recurrence is simply called recurrence by Lamperti in \cite{L60,L62,L63}.  
\end{remark}

\begin{proof}
Let $Y_m = X_m^{(2-\vartheta)/2}$.  Since $(Y_m,m\geq 0)$ is Markov, \cite[Lemma 7.1]{L62} shows that the claims of recurrence and transience   for $(Y_m,m\geq 0)$ are settled by \cite[Theorem 3.2]{L60}.  Assumptions 1, 2, and 3 and \cite[Lemma 7.1]{L62} show that the hypotheses of Theorem 4.1 in \cite{L63} are satisfied for $(Y_m,m\geq 0)$.  Combining the conclusions of \cite[Theorem 4.1]{L63} with Assumptions 1 and 3 shows that the hypotheses of Theorem 5.1 in \cite{L62} are satisfied so our claim follows from the conclusion of \cite[Theorem 5.1]{L62} along with translating the results for $(Y_m,m\geq 0)$ back to $(X_m,m\geq 0)$.
\end{proof}

We note that the functional limit theorem of \cite[Theorem 5.1]{L62} actually pertains to the scaled linearly interpolated process rather than the scaled step process, and convergence in distribution on $C(\R_+,\R)$, but the convergence of the scaled step process in distribution on $D(\R_+,\R)$ follows immediately.

In our present context, there is no difference between $A$-recurrence and neighborhood recurrence.

\begin{proposition}\label{prop A-to-eps}  
If $(X_{k,d})_{k\geq 0}$ is a Markov process with transition operator \eqref{eq transition dth}, then $(X_{k,d})_{k\geq 0}$ is neighborhood recurrent if and only if $(|X_{k,d}|)_{k\geq 0}$ is $A$-recurrent for some $A \geq 0$.
\end{proposition}

\begin{proof}
Fix $\eps >0$ and observe from \eqref{eq transition dth} that $\min_{0\leq x\leq A} \P_x(|X_{1,d}| < \eps) >0$.  Combined with the strong Markov property this implies that $\P_0(|X_{k,d}| < \eps \ i.o.) =1$ since $\P_0(|X_{k,d}| \leq A \ i.o.) =1$.
\end{proof}

With Theorem \ref{theorem Lamperti} in hand, the idea of our proofs is essentially straightforward, but the calculations become quite involved in the general case.  Thus, before getting into the true model, we show how the method works in a simplified model where $\lambda=0$ and the particle travels distance exactly equal to one between reflections. 

\subsection{Motion with deterministic distance between reflections}\label{section distance 1}

This section is a warm up, in the sense that we analyze a simplified model, to develop a sense for results that we can expect in a more realistic and hence more complicated situation. Specifically, we assume that the 
distance between any two consecutive reflections measured along the trajectory of the particle is exactly one.  In this model, upon reflection at $\bx \in \R^d$, the particle starts its path in a uniform direction $\bu \in \bS^{d-1}$ and then travels along the parabola \eqref{equation basic parabola 1} (with $t$ measuring the time since the last reflection) until it has traveled distance exactly one, at which point it reflects again.  Let $(\bX(t), t\geq 0)$ be the path of such a particle and let the discrete time process $(X_d^*(k), k\in \N_0)$ record the positions of $(X_d(t),t\geq 0)$ at the reflection times. Note that this is not the same as sampling of $X_d$ at equal or identically distributed time intervals because the velocity of $X_d$ increases with $|X_d|$ and the times between scattering events become smaller on average. 
The process $(X_d^*(k), k\in \N_0)$ is a Markov chain with transition operator $U$ that acts on $C^2$ function $f$ with compact support in $(-\infty,0)$ by
\[ (U f)(y) = \int_{\bS^{d-1}} f(\Lambda_d(y\be_d , \bu , t(y\be_d,\bu))) \sigma(d\bu)\]
where $t(\bx,\bu)$ is the time it takes to travel distance 1 along the parabola in \eqref{equation basic parabola 1} with initial position $\bx$ and initial velocity in the direction of $\bu$.  That is, $t(\bx,\bu) = \inf\{s : \ell(\bx,\bu,s) >1\}$ where 
\[\ell(\bx,\bu,t) = \int_0^t \sqrt{\vsd (1-u_d^2) + \left(\vd u_d-gs\right)^2}\ ds.\]

\begin{theorem}\label{theorem det-dist}
The process $(X_d^*(m), m\in \N_0)$ is neighborhood recurrent if $d\leq 3$ and transient if $d \geq 4$. 
\end{theorem}

\begin{proof}
We will apply  Theorem \ref{theorem Lamperti}
to the process $(|X_d^*(m)|, m\in \N_0)$ in place of $(X_m,m\geq 0)$, with $\vartheta=0$.
Conditions $2$ and $3$ of Theorem \ref{theorem Lamperti} are easy to check, leaving the problem of finding the limits in part $1$.  In order to apply Theorem \ref{theorem Lamperti} we need to analyze $\mu_1(y)=|y| \E_y( X^*_d(1) - y)$ and $ \mu_2(y)=\E_y\left[(X^*_d(1) - y)^2\right] $ as $y$ tends to $-\infty$.  The key to doing this is to analyze how $t(y\be_d,\bu)$, the time between reflections, depends on $y$.  It is easy to check that we have the monotonicity relation $t(\bx,-\be_d)\leq t(\bx,\bu)\leq t(\bx,\be_d)$, which is intuitive because it takes the longest to travel straight up and the shortest to travel straight down.  Moreover, assuming $x_d\leq -1$, as we will for the remainder, we can explicitly compute 
\[t(\bx,\be_d) = \sqrt{\frac{2}{g}} \left(\sqrt{|x_d|} - \sqrt{|x_d|-1}\right) 
\quad \textrm{and} \quad t(\bx,-\be_d) = \sqrt{\frac{2}{g}} 
\left(\sqrt{|x_d|+1} - \sqrt{|x_d|}\right).
\]
From this, one observes that $\sqrt{|y|}t(y\be_d,\pm \be_d) \to (\sqrt{2g})^{-1/2}$ 
as $y\to -\infty$ and, consequently,
\begin{equation}\label{equation time scale} 
\lim_{y\to -\infty} \sqrt{|y|}t(y\be_d, \bu) \to \frac{1}{\sqrt{2g}},
\end{equation}
uniformly in $\bu$. 
Let $\ell_t(\bx,\bu,t), \ell_{tt}(\bx,\bu,t)$ and $ \ell_{ttt}(\bx,\bu,t)$ denote the first, second and third partial derivatives, resp., of $\ell(\bx,\bu,t)$ in the third variable.
We have $\ell_t(\bx,\bu,0) = \vd$ and $\ell_{tt}(\bx,\bu,0) = -gu_d$. 
It follows from the definition of $t(\bx,\bu)$ that $\ell(\bx,\bu ,t(\bx, \bu))=1$.  Taylor expanding $\ell$ in the $t$ variable yields
\[ \begin{split} 1& = \ell(\bx,\bu,t(\bx, \bu)) \\
&=\vd t(\bx, \bu) - \frac{gu_d}{2} t(\bx, \bu)^2 + \frac{\ell_{ttt}(\bx,\bu,\alpha)}{6}t(\bx, \bu)^3
\end{split}\]
for some $\alpha =\alpha(\bx,\bu) \leq t(\bx,\be_d)$.  Rearranging, this yields the relation
\begin{equation}\label{eq ell expansion} t(\bx, \bu) = \frac{1}{\vd}\left(1 + \frac{gu_d}{2} t(\bx, \bu)^2 - \frac{\ell_{ttt}(\bx,\bu,\alpha)}{6}t(\bx, \bu)^3\right).\end{equation}

We have
\begin{equation}\label{eq expectation split} 
|y| \E_y( X^*_d(1) - y) = \E_y\left(U_d\sqrt{2g} |y|^{3/2}t(y\be_d, \bU)\right) 
+ \E_y\left(-\frac{g}{2}|y|t(y\be_d, \bU)^2\right).
\end{equation}
By \eqref{equation time scale}, the second term in \eqref{eq expectation split} converges to $-1/4$ as $y\to -\infty$.  To analyze the first term, we substitute \eqref{eq ell expansion} and use  $\E_y(U_d)=0$ to find that
\begin{equation}\label{eq expectation split 2} 
\E_y\left(U_d\sqrt{2g} |y|^{3/2}t(y\be_d, \bU)\right) =  \E_y\left(\frac{gU^2_d}{2}|y| t(y\be_d, \bU)^2\right) - \E_y\left(U_d\frac{\ell_{ttt}(y\be_d,\bU,\alpha(y,\bU))}{6}|y|t(y\be_d,\bU)^3\right).
\end{equation}

From \eqref{equation time scale} we see that 
\[\lim_{y\to-\infty}  \E_y\left(\frac{gU^2_d}{2}|y| t(y\be_d, \bU)^2\right) = \frac{1}{4} \E(U_d^2) = \frac{1}{4d}.\]
Furthermore, straightforward but tedious calculations show that $\ell_{ttt}(y\be_d,\bu,t) = O(|y|^{-1/2})$ as $y\to -\infty$ uniformly in $\bu$, and $0\leq t\leq t(y\be_d,\be_d) $ which, combined with \eqref{equation time scale}, show that the second term in \eqref{eq expectation split 2} converges to $0$ as $y\to -\infty$.  Therefore
\[-a:=\lim_{y\to-\infty}|y|\mu_1(y)= \lim_{y\to -\infty} |y| \E_y( X^*_d(1) - y) = \frac{1}{4d} - \frac{1}{4} = \frac{1-d}{4d}.\]
The negative sign is because Lamperti's processes are positive while ours are negative.  Similarly, using \eqref{equation time scale} we see that
\[ 
b:= \lim_{y\to-\infty} \mu_2(y)= \lim_{y\to-\infty} \E_y\left[(X^*_d(1) - y)^2\right] 
=  \lim_{y\to-\infty} \E_y\left[\left(U_d\sqrt{2g|y|} t(y\be_d, \bU)
-\frac{g}{2}t(y\be_d, \bU)^2 \right)^2\right] = \frac{1}{d}.
\]
This shows that the limits in Condition 1 of Theorem \ref{theorem Lamperti} exist.  Moreover, 
\[ 2a-b= \frac{d-1}{2d} - \frac{1}{d} = \frac{d-3}{2d}.
\]
The claims of Theorem \ref{theorem det-dist} can now be read off from Theorem \ref{theorem Lamperti}.  Since $2a-b$ is positive if $d\geq 4$, the process is transient in this case.  Moreover, $2a-b$ is negative if $d\leq 2$ so the process is $A$-recurrent in this case for some $A\geq 0$. In the case $d=3$, we have $2a-b=0$, so this is the critical case.  One can verify that when $d=3$,  $2|y| \mu_1(y)-\mu_2(y) = O(|y|^{-\epsilon})$ for sufficiently small $\epsilon>0$ and, consequently, the process is $A$-recurrent in this case as well.  We leave this calculation in the present toy model to the reader since we do the analogous (more difficult) calculation for our main model below.  A straightforward argument using the Markov property as in Proposition \ref{prop A-to-eps} shows that $A$-recurrence for any $A\geq 0$ implies neighborhood recurrence for $(|X_d^*(m)|, m\in \N_0)$.
\end{proof}

\section{The general model}\label{highdimpower}
In this section we address the general model with generator \eqref{eq generator d} where $h$ is of the form $h(y)=c|y|^\lambda$ for some $\lambda \geq 0$ and $c>0$.  We prove some limit theorems and results on transience and recurrence. Although Section \ref{se method} illustrates our methods, the results in this section are technically more difficult because we must control the distance the particle travels between reflections as well as the time between reflections in order to establish our invariance principles.

\subsection{Basic Estimates}\label{se basic estimates}
Recall that for $y\leq 0$, we define $N(y,\bu)= N(y\be_d,\bu)$ where $N(\bx,\bu)$ is defined in \eqref{eq inter reflection} for $\bx\in\R^d$.

\begin{lemma}\label{lemma large y distribution}
For every $t\geq 0$ we have
\[\lim_{y\to -\infty} \sup_{\bu\in \bS^{d-1}} \left|\P\left(\v h(y)N(y,\bu) >t\right) - e^{-t}\right|=0.\] 
\end{lemma}

\begin{proof}
We use \eqref{eq inter reflection} and the substitution $w = \v h(y) s$ to see that
\begin{align*} 
-\log&\left(\P\left(\v h(y)N(y,\bu) >t\right)\right) 
=
-\log\left(\P\left(N(y,\bu) >\frac{t}{\v h(y)}\right)\right) 
\\  
&= \int_0^{t} \frac{ h\left(y+\frac{u_d}{h(y)}w - \frac{1}{4|y|h(y)^2}w^2\right)}{\v h(y)} \sqrt{\vs (1-u_d^2)+\left(\v u_d-\frac{g}{\v h(y)}w\right)^2}dw  \\
&= \int_0^{t} \frac{h\left(y+\frac{u_d}{h(y)}w - \frac{1}{4|y|h(y)^2}w^2\right)}{h(y)} \sqrt{ (1-u_d^2)+\left(u_d-\frac{1}{2|y| h(y)}w\right)^2}dw.
\end{align*}
For $h(y)=c|y|^\lambda$, we have
\[ \lim_{y\to-\infty} \sup_{(w,\bu)\in [0,t]\times\bS^{d-1}}\left|  \frac{h\left(y+\frac{u_d}{h(y)}w - \frac{1}{4|y|h(y)^2}w^2\right)}{h(y)} - 1\right|=0,\]
and the lemma follows.
\end{proof}

In fact, this convergence in distribution can be extended to convergence of moments.

\begin{lemma} \label{lemma large y bound}
For fixed $p\geq 1$,
\[\E \left[ \left(\max\{\v h(y),1\}N(y,\bu)\right)^p\right] \]
is bounded uniformly in $y$ and $\bu \in \bS^{d-1}$.
\end{lemma}

\begin{proof}
We handle the cases $y\leq -1$ and $y>-1$ separately.  For $-1\leq y\leq 0$ 
there is a finite longest time for a parabolic path started with $-1\leq y\leq 0$ to leave $[-1,0]$.  Outside this interval 
$h$ is bounded below by a strictly positive constant. Hence, once the particle is outside $[-1,0]$, it will encounter a scatterer at some strictly positive rate. This implies that all of the $N(y,\bu)$ with $-1\leq y\leq 0$ are stochastically dominated by a single random variable with an exponential tail. The lemma easily follows in this case.  

We now turn to the case $y\leq -1$.  A monotonicity argument shows that 
\[ \P(N(y,\bu) > t) \leq \exp\left[-\int_0^t c\left( \frac{g}{2}s^2 -  \v s - y\right)^\lambda\left|\v - g s\right| ds \right].
\]
Fix $0< \eps < 1/4$.  We need to control the amount of time the particle can spend above $\epsilon$, since this is where the collision rate is low and $\P(N(y,\bu) > t)$ decreases slowly in this region.  Define 
\begin{align}\label{s22.5}
 s_-(\bu) = \inf\left\{ s \geq 0 : y+u_d\v s - \frac{g}{2}s^2 =-\eps \right \} 
\end{align}
and
\[  s_+(\bu) = \sup\left\{ s \geq 0 : y+u_d\v s - \frac{g}{2}s^2 =-\eps \right \} .\]
Monotonicity arguments show that
\begin{align}\label{s22.6} 
s_-(\bu) \geq s_-(\be_d) = \sqrt{\frac{2}{g}}\left(\sqrt{|y|}- \sqrt{\eps}\right)
\end{align}
and
\[ s_+(\bu) \leq s_{+}(\be_d) = \sqrt{\frac{2}{g}}\left(\sqrt{|y|}+ \sqrt{\eps}\right).\]
To simplify notation, let us use $s_\pm:=s_\pm(\be_d)$.  This leads to the bounds
\begin{equation}\label{eq h tail bounds}  \P(N^n(y,\bu) >t) \leq \begin{cases}
\vspace{.2cm} \exp\left[-h(-\epsilon) t \left(\v  - \frac{g}{2} t\right)\right], & t\leq s_- ,\\
\exp\left[ -h(-\epsilon) \left(\v [s_-+s_+-t] + \frac{g}{2}[t^2-s_-^2-s_+^2] \right)\right], & t\geq s_+.
 \end{cases}
\end{equation}

An application of Fubini's theorem shows that
$\E(R^p)
=  p\int_0^\infty t^{p-1} \P(R>t) dt$
for any non-negative random variable $R$.
Using this and \eqref{eq h tail bounds} we find that
\begin{align}\label{d24.3}
\E(N(y,\bu)^p) &\leq p\int_0^{s_-/4}t^{p-1} \exp\left[-\int_0^t c\left( \frac{g}{2}s^2 -  \v s - y\right)^\lambda\left(\v - g s\right)ds\right]dt \\
&\quad+ p\int_{s_-/4}^{4s_+} t^{p-1}\exp\left[-h(-\epsilon) s_- \left(\v  - \frac{g}{8} s_-\right)/4\right] dt \nonumber\\
 &\quad+ p\int_{4s_+}^\infty t^{p-1}\exp\left[ -h(-\epsilon) \left(\v \left[s_-+s_+-t\right] + \frac{g}{2}\left[t^2-s_-^2-s_+^2\right] \right)\right] dt.\nonumber
\end{align}
The first integral is the most challenging, so we take care of the second and third integrals first.  Since
 \begin{equation*}
\v  - \frac{g}{8} s_- =  \v - \frac{\sqrt{2g}}{8} (\sqrt{|y|} - \sqrt{\epsilon}) \geq \frac{3}{4} \v,
\end{equation*}
 we have
\begin{align}\label{d24.2}
 \lim_{y\to-\infty} p \left(\v h(y)\right)^p\int_{s_-/4}^{4s_+} t^{p-1}\exp\left[-h(-\epsilon) s_- \left(\v  - \frac{g}{8} s_-\right)/4\right] dt =0
\end{align}
because the integral term decays exponentially in $|y|$.  Similarly we have
\begin{align}\label{d24.1}
\lim_{y\to-\infty} p \left(\v h(y)\right)^p \int_{4s_+}^\infty t^{p-1}\exp\left[ -h(-\epsilon) \left(\v \left[s_-+s_+-t\right] + \frac{g}{2}\left[t^2-s_-^2-s_+^2\right] \right)\right] dt =0.
\end{align}
For the first integral in \eqref{d24.3}, 
use the Mean Value Theorem to see that for $0\leq t \leq s_-/4$,
\begin{align*}  
-\int_0^t &c\left( \frac{g}{2}s^2 -  \v s - y\right)^\lambda\left(\v   - g s\right)ds 
 =\frac{c}{\lambda+1}\left[\left( \frac{g}{2}t^2 -  \v t - y\right)^{\lambda+1} - |y|^{\lambda+1}\right] \\
& \leq \frac{c}{\lambda+1}\left[\left(  - \frac{ \v}{2} t - y\right)^{\lambda+1} - |y|^{\lambda+1}\right]
\leq - c \frac{ \v}{2} t \inf\left\{|z|^\lambda:  - \frac{ \v}{2} t - y \leq z \leq -y\right\}
\\
& 
\leq - c \frac{ \v}{2} t  \left| - \frac{ \v}{2} s_+/4 - y\right|^\lambda
\leq  - C \v \left| y\right|^\lambda t
,\end{align*} 
where $C>0$ is a constant depending on $\lambda$ but not $y$.  Consequently, we have
\begin{align*}
p\int_0^{s_-}t^{p-1} &\exp\left[-\int_0^t c\left( \frac{g}{2}s^2 -  \v s - y\right)^\lambda\left(\v - g s\right)ds\right]dt \\
& \leq p\int_0^\infty t^{p-1} \exp\left( - C \v \left| y\right|^\lambda t\right)dt  = \frac{p!}{ C^p (\vs)^{p/2} |y|^{p\lambda}}.
\end{align*}
This and \eqref{d24.3}-\eqref{d24.1} prove the result for $y\leq -1$.  
\end{proof}

The next lemma gives a uniform version of the classical result that convergence in distribution together with bounded moments implies the convergence of moments. 

\begin{lemma}\label{lemma moments monomial}
For every $p\geq 1$ we have
\[\lim_{y\to -\infty} \sup_{\bu\in \bS^{d-1}} \left|\E \left[ \left(\v h(y)N(y,\bu)\right)^p\right] - \int_0^\infty pt^{p-1}e^{-t}dt\right|=0.\] 
\end{lemma}

\begin{proof}
Let $B(y,\bu,r) = \left\{ \v h(y)N(y,\bu) \leq r\right\}$ and define $q$ by $1/q + p/(p+1) = 1$.  Also, let $\eta(y) = \v h(y)$. Then
\[ \begin{split}
\E \left(\left[\eta(y)^pN(y,\bu)^p\right] - \left[\eta(y)^pN(y,\bu)^p\right] \land r^p  \right)
&\leq 
\E \left( \left[\eta(y)^pN(y,\bu)^p\right] \cf_{B^c(y,\bu,r)} \right) 
\\
&\leq
\E \left( \eta(y)^{p+1}N(y,\bu)^{p+1}\right)^{p/(p+1)}  \left(\P( \eta(y) N(y,\bu) > r )\right)^{1/q}\\
&\leq 
\E \left( \eta(y)^{p+1}N(y,\bu)^{p+1}\right)^{p/(p+1)}\E(  \eta(y) N(y,\bu)   )^{1/q} r^{-1/q}.
\end{split}\]
Both expectations are uniformly bounded by Lemma \ref{lemma large y bound} so the bound goes uniformly to 0 as $r\to\infty$.  The proof is completed by noting that it follows from Lemma \ref{lemma large y distribution} that for every $r\geq 0$,
\[ \lim_{y\to-\infty}\sup_{\bu\in \bS^{d-1}} \left|\E \left[ \left(\v h(y)N(y,\bu)\right)^p\wedge r^p\right] - \int_0^rpt^{p-1}e^{-t}dt\right|=0. \qedhere\]
\end{proof}

The next lemma is needed to control lower order fluctuations.  This is where the averaging occurs and it becomes important that the scattering distribution has mean $0$.

\begin{lemma}\label{de16.6}
Let $\bU$ be uniformly distributed on $\bS^{d-1}$ and let $N(y,\bU)$ be distributed like $N(y,\bu)$ conditional on $\bU=\bu$.  We then have 
\[ \lim_{y\to-\infty} \left| \left( h(y)^2\sqrt{2g|y|^3}\right) \E\left[ U_d N(y,\bU) \right] - \frac{1+2\lambda}{2d} \right| =0.\]
\end{lemma}

\begin{proof}
Let  
\begin{align*}
H(y,\bu,t) &= h\left(y+u_d\v t -\frac{g}{2}t^2\right),\\
M(y,\bu,t) &= \sqrt{\vs (1-u_d^2)+\left(u_d\v-gt\right)^2},\\
F(y,\bu,t) & = \int_0^t H(y,\bu,s)M(y,\bu,s)  ds.
\end{align*}
Using the change of variables $z= F(y,\bu,t)$ and the density of $N(y,\bu)$ derived from \eqref{eq inter reflection} one finds that 
\begin{align}\label{s9.1}
 1= \E[ F(y,\bu, N(y,\bu))] \quad \textrm{and}\quad 2= \E[ F(y,\bu, N(y,\bu))^2] ,
\end{align}
for all $y< 0$ and $\bu \in \bS^{d-1}$.
Taylor expanding $F$ in $t$ about $0$, we find that for $t< (2|y|/g)^{1/2}$,
\begin{equation}\label{eq F_taylor monomial} F(y,\bu, t) = h(y)\v t + F''(y,\bu, T(y,\bu,t)) \frac{t^2}{2}, \end{equation}
for some $0\leq T(y,\bu, t) \leq t$.  Let $B(y,\bu) = \{N(y,\bu) < 1\}$.  We then have
\begin{equation}\label{eq decomp monomial} 1= \E[ F(y,\bu, N(y,\bu))]  = \E[ F(y,\bu, N(y,\bu)) \cf_{B}] + \E[ F(y,\bu, N(y,\bu))\cf_{B^c}]\end{equation}
and, by the Cauchy-Schwarz inequality and \eqref{s9.1},
\[  \E[ F(y,\bu, N(y,\bu))\cf_{B^c}] \leq \sqrt{2 \P(N(y,\bu) \geq 1)}. \]
By Lemma \ref{lemma large y bound} we see that for every $r\geq 0$
\[  \lim_{y\to-\infty} \sup_{\bu \in \bS^{d-1}} |y|^r \P(N(y,\bu) \geq 1) =0.\]
Consequently 
\begin{equation}\label{eq F-1 monomial}
 \lim_{y\to-\infty} \sup_{\bu \in \bS^{d-1}} |y|^r  \E[ F(y,\bu, N(y,\bu))\cf_{B^c}]  = 0 .
\end{equation}
Similarly, for every $r\geq 0$ and $p\geq 1$ we see that 
\begin{equation}\label{eq M-1 monomial} 
\lim_{y\to-\infty} \sup_{\bu \in \bS^{d-1}} |y|^{r}\E \left[ N(y,\bu)^p \cf_{B^c}\right] = 0.
\end{equation} 
For $y$ such that $y\leq -g/2$, substituting \eqref{eq F_taylor monomial} into the first expectation on the right hand side of \eqref{eq decomp monomial} and solving for $\E( N(y,\bu) \cf_{B})$ yields
\begin{align*}  \E & (N(y,\bu) \cf_{B})\\ 
&=  \frac{1}{h(y)\v } \left(1  -  \E[ F(y,\bu, N(y,\bu))\cf_{B^c}]  
 - \frac{1}{2}\E\left[F''(y,\bu, T(y,\bu,N(y,\bu)))N(y,\bu)^2\cf_{B} \right]\right).
\end{align*}
Conditioning $\E (U_dN(y,\bU) \cf_{B})$ on $\{\bU=\bu\}$ and using the fact that $\E(U_d)=0$, we have
\begin{align} \label{d27.1}
h(y)^2\sqrt{2g|y|^3} \E(U_dN(y,\bU)\cf_{B}) 
 = &-h(y)|y|  \E[ U_d F(y,\bU, N(y,\bU))\cf_{B^c}] \\
& - \frac{h(y)|y|}{2} \E\left[U_dF''(y,\bU, T(y,\bU,N(y,\bU)))N(y,\bU)^2\cf_{B} \right] . \nonumber
\end{align}
The first term on the right hand side vanishes as $y\to-\infty$ by \eqref{eq F-1 monomial}.  For the second term, observe that
\begin{align*} 
F''(y,\bu,t) & =  \ H'(y,\bu,t)M(y,\bu,t) +H(y,\bu,t)M'(y,\bu,t) \\
&=  \left(u_d\v  -gt\right)h'\left(y+u_d\v t -\frac{g}{2}t^2\right)M(y,\bu,t)\\
& \quad -  h\left(y+u_d\v t -\frac{g}{2}t^2\right) \frac{g\left(\v u_d-gt\right) }{ \sqrt{\vs  (1-u_d^2)+\left(\v u_d-gt\right)^2}}.
\end{align*}
Elementary calculations show that 
\[\lim_{y\to-\infty} \sup_{(t,\bu)\in [0,1]\times \bS^{d-1}} \left| \frac{ F''(y,\bu,t)}{h(y)} -gu_d(-2\lambda -1) \right| =0.\] 
Therefore, a combination of Lemma \ref{lemma moments monomial} and \eqref{eq M-1 monomial} shows that
\[\lim_{y\to\infty} \frac{h(y)|y|}{2} \E\left[U_dF''(y,\bU, T(y,\bU,N(y,\bU)))N(y,\bU)^2\cf_{B} \right] = - \frac{1+2\lambda }{2} \E(U_d^2) =   -\frac{1+2\lambda }{2d}. \]
The lemma follows by combining this with \eqref{d27.1}.
\end{proof}

\begin{proposition}\label{proposition asymptotic mean variance}
Let $(Y_m,m\geq 0)$ be the Markov chain with transition operator \eqref{eq transition dth} and $h(y)=c|y|^\lambda$.  For $x\geq 0$, define
\[ \hat\mu_k(x) = \E\left[ (|Y_1| - x)^k \big| \ Y_0=-x\right] .\]
We then have $\sup_x x^{\lambda k}\mu_k(x) <\infty$ for all $k\geq 1$ and  
\[\lim_{x\to\infty} x^{1+2\lambda}\hat\mu_1(x) =\frac{d-1-2\lambda}{2dc^2} \quad \textrm{and} \quad \lim_{x\to\infty}x^{2\lambda} \hat\mu_2(x) = \frac{2}{dc^2}.\]
\end{proposition}

\begin{proof}
First note that, by \eqref{equation basic parabola 1} and \eqref{eq transition dth},
\begin{align*}
x^{\lambda k}\hat\mu_k(x) &=  \sum_{i=0}^k (-1)^{k-i}\left(\frac{g}{2}\right)^i (2gx)^{(k-i)/2} x^{\lambda k} \E\left[ U_d^{k-i} N(-x,\bU)^{i+k}\right] \\ 
&\leq  \sum_{i=0}^k\left(\frac{g}{2}\right)^i (2gx)^{(k-i)/2} x^{\lambda k} \E\left[N(-x,\bU)^{i+k}\right],
\end{align*}
which is bounded, when $x\to \infty$, by Lemma  \ref{lemma large y bound}.  Observe that 
\[x^{1+2\lambda} \hat\mu_1(x) = \E\left( \frac{gxh(-x)^2}{2c^2} N(-x,\bU)^2 - \frac{U_d\sqrt{2gx^3} h(-x)^2}{c^2} N(-x,\bU) \right).\]
Lemma \ref{lemma moments monomial} implies that 
\[ \lim_{x\to\infty}  \E\left((1/2) g xh(-x)^2N(-x,\bU)^2\right) = 1/2,\]
while Lemma \ref{de16.6} shows that 
\[ \lim_{x\to\infty} \E\left( U_d\sqrt{2gx^3} h(-x)^2 N(-x,\bU) \right) = \frac{1+2\lambda}{2d}.\] 
Consequently,
\[\lim_{x\to\infty} x^{1+2\lambda} \hat\mu_1(x) = \frac{1}{2c^2}-\frac{ 1+2\lambda}{2dc^2} =\frac{d-1- 2\lambda}{2dc^2}  .\]
Similarly, we see that
\begin{align}\label{d28.1}
 \hat\mu_2(x) & = \E\left[\left( \frac{g}{2}N(-x,\bU)^2- U_d\sqrt{2gx} N(-x,\bU) \right)^2 \right]\\
& = \E\left[ 2gx U_d^2 N(-x,\bU)^2 - gU_d\sqrt{2gx} N(-x,\bU)^3 +  \frac{g^2}{4}N(-x,\bU)^4\right].\nonumber
\end{align} 
Lemma  \ref{lemma moments monomial} implies that
\begin{align}\label{d28.2}
\lim_{x\to\infty}
\E\left[  - gU_d\sqrt{2gx} N(-x,\bU)^3 +  \frac{g^2}{4}N(-x,\bU)^4\right] =0,
\end{align}
and
\begin{align*}
\lim_{x\to\infty}
\E\left[x^{2\lambda } 2gx U_d^2 N(-x,\bU)^2\right] =\frac{2}{dc^2}.
\end{align*}
This, \eqref{d28.1} and \eqref{d28.2} yield
$x^{2\lambda } \hat\mu_2(x) \to \frac{2}{dc^2}$ as $x\to\infty$.
\end{proof}

The next proposition contains an estimate needed in the case when $d=3$ and $\lambda=0$.

\begin{proposition}\label{prop recurrence3d}
If $d=3$ and $\lambda=0$ then $2x\hat\mu_1(x) - \hat\mu_2(x) \leq O(x^{-\delta})$ for some $\delta>0$ as $x\to \infty$.
\end{proposition}

\begin{proof}
For simplicity, we assume $c=1$; the proof is similar in other cases.  Observe that
\begin{align}\label{d30.1}
2x \hat \mu_1(x) -\hat \mu_2(x) &= \E\left( gx(1-2U_3^2)N(-x,\bU)^2\right) - \E\left(U_3\sqrt{8gx^3}  N(-x,\bU) \right)\\
&\quad  +\E \left(  gU_3\sqrt{2gx} N(-x,\bU)^3\right) - \E\left( \frac{g^2}{4}N(-x,\bU)^4 \right).\nonumber
\end{align}
If $\delta < 1$, then
\begin{align}\label{d30.2}
 \lim_{x\to\infty} x^{\delta}\E \left(  gU_3\sqrt{2gx} N(-x,\bU)^3\right) =0 \quad \textrm{and}\quad \lim_{x\to\infty} x^{\delta} \E\left( \frac{g^2}{4}N(-x,\bU)^4 \right) =0,
\end{align}
by Lemma \ref{lemma moments monomial}.

For the remaining terms we need more careful estimates. 
Consider any $\eps\in(0,1/2)$, $ r>1$ and $p\geq 1$. Let $q$ be so large that
$-p/2 - q\eps/2 < -r$.
By Lemma \ref{lemma large y bound}, for all $\bu$ and large $x$,
\begin{align*}
\E&\left[ N(-x,\bu)^p \cf_{\left\{ \sqrt{x}N(-x,\bu) > x^\epsilon\right\}}\right]
\leq 
\E\left[ N(-x,\bu)^{2p}\right]^{1/2}
\E\left[
 \cf_{\left\{ \sqrt{x}N(-x,\bu) > x^\epsilon\right\}}\right]^{1/2}\\
&\leq c_1 x^ {-p/2} 
\P\left( (\sqrt{x}N(-x,\bu))^q > x^{q\epsilon}\right)^{1/2}
\leq c_1 x^ {-p/2} 
\left(\E\left[ (\sqrt{x}N(-x,\bu))^q\right] / x^{q\epsilon}\right)^{1/2}
\leq c_2 x^ {-p/2}  x ^{-q\eps/2}.
\end{align*}
It follows that
\begin{equation} \label{eq epsilon} \lim_{x\to \infty} \sup_{\bu\in\bS^{d-1}} x^r \E\left[ N(-x,\bu)^p \cf_{\left\{ \sqrt{x}N(-x,\bu) > x^\epsilon\right\}}\right] =0.
 \end{equation}
It is easy to see that for all $\bu\in\bS^{2}$ we have 
\[ \P(N(-x,-\be_3)>t) \leq \P(N(-x,\bu)>t)\leq \P(N(-x,\be_3)>t).\]
Consequently, for sufficiently large $x$ we have
\[\begin{split} \E\left( gxN(-x,\bU)^2\right) & \leq  \E\left( gxN(-x,\be_3)^2\right)\\
& =   \E\left( gxN(-x,\be_3)^2 \cf_{\left\{ \sqrt{x}N(-x,\be_3) \leq x^\epsilon\right\}}\right) +  \E\left( gxN(-x,\be_3)^2 \cf_{\left\{ \sqrt{x}N(-x,\be_3) > x^\epsilon\right\}}\right) \\
&= 2g \int_0^{x^{\epsilon}} t \P( \sqrt{x} N(-x,\be_3)>t) dt + o(x^{-r})\\
& =  2g \int_0^{x^{\epsilon}} t \exp\left(-\int_0^{tx^{-1/2}} \left|\sqrt{2gx}-gs\right| ds \right)dt + o(x^{-r}) \\
& =  2g \int_0^{x^{\epsilon}} t \exp\left(-\sqrt{2g}t+\frac{gt^2}{x} \right)dt + o(x^{-r}) \\
& =  2g \int_0^{x^\epsilon} t \exp\left(-\sqrt{2g}t\right)\left(\exp\left(\frac{gt^2}{x}\right) -1 \right)dt + 2g\int_0^{x^\epsilon} t \exp\left(-\sqrt{2g}t\right) dt + o(x^{-r}) \\
&\leq O(x^{-1}) +1 + o(x^{-r}) 
.\end{split}\]
Similarly, for some $\nu>0$
\[\begin{split}  \E\left(2gxU_3^2N(-x,\bU)^2\right) & \geq  \frac{2gx}{3} \E\left(N(-x,-\be_3)^2\right) \\
 & = \frac{4g}{3} \int_0^\infty t \exp\left(-\sqrt{2g}t-\frac{gt^2}{x} \right)dt \\
 & \geq (e^{-gx^{2\epsilon-1}}-1) \frac{4g}{3} \int_0^{x^\epsilon} t \exp\left(-\sqrt{2g}t \right)dt + \frac{4g}{3} \int_0^{x^\epsilon} t \exp\left(-\sqrt{2g}t \right)dt \\
 & \geq O(x^{2\epsilon -1}) + o(e^{-\nu x^{\epsilon}}) + \frac{2}{3}
.\end{split} \]
Combining the last two estimates, we obtain, 
\begin{align}\label{d30.3}
\E\left( gx(1-2U_3^2)N(-x,\bU)^2\right) \leq \frac{1}{3} + O(x^{2\epsilon -1}).
\end{align}
Using \eqref{eq epsilon} and arguing as in the proof of Lemma \ref{de16.6} (and using the notation there) we have 
\begin{align*} 
- \E&\left(U_3\sqrt{8gx^3}  N(-x,\bU) \right)
=- \sqrt{8gx^3} \E\left(U_3 N(-x,\bU)\cf_{\{ \sqrt{x}N(-x,\bU)\leq x^\epsilon\}} \right) + o(x^{-r}) \\
& = -x \E\left[ U_3 F''(-x,\bU,T(-x,\bU,N(-x,\bU))) N(-x,\bU)^2 \cf_{\{ \sqrt{x}N(-x,\bU)\leq x^\epsilon\}} \right] + o(x^{-r}),
\end{align*}
where $0\leq T(-x,\bU,N(-x,\bU)) \leq x^{\epsilon - 1/2}$ and 
\[ F''(-x,\bu,t) = \frac{ g\left( u_3 - \frac{gt}{\sqrt{2gx}} \right)}{ \sqrt{1-u_3^2 + \left(u_3-\frac{gt}{\sqrt{2gx}}\right)^2}} = J\left(\bu, \frac{gt}{\sqrt{2gx}}\right) \]
with
\[ J(\bu,t) := \frac{ g\left( u_3 - t \right)}{ \sqrt{1-u_3^2 + \left(u_3-t\right)^2}} .\]
Note that, for sufficiently small $T$, $J(\bu,t)$ is continuously differentiable on $\bS^2\times [-T,T]$ and, consequently, there exists a constant $C$ such that
\[ | J(\bu,t) - gu_3| = | J(\bu,t) - J(\bu,0)| \leq C |t|.\] 
Therefore
\[\begin{split} 
- \E\left(U_3\sqrt{8gx^3}  N(-x,\bU) \right) &= -gx \E\left[ U^2_3  N(-x,\bU)^2 \cf_{\{ \sqrt{x}N(-x,\bU)\leq x^\epsilon\}} \right] 
+ O(x^{\epsilon -1}) + o(x^{-r}) \\
& \leq -\frac{1}{3} +O(x^{2\epsilon -1}) + o(e^{-\nu x^{\epsilon}})+O(x^{\epsilon -1}) + o(x^{-r})
.\end{split}\]
This, \eqref{d30.1}, \eqref{d30.2} and \eqref{d30.3} imply that $2x\hat \mu_1(x) -\hat \mu_2(x) \leq O(x^{-\delta})$ for every $0<\delta <1$.
\end{proof}

\section{Power function  scatterer density: proofs of the main results}\label{se proofs}

\begin{proof}[Proof of Theorem \ref{theorem power function intro t/r}]

Let $(\mathbf{X}_k)_{k\geq 0}= \{(X_{1,k},\dots,X_{d-1,k},X_{d,k})\}_{k\geq 0}$ be the Markov chain with transition operator \eqref{eq transition} started from $\mathbf{0}$, with gravitation $g$ and scatterer density $h(\mathbf{x})=h(x_d)=c|x_d|^\lambda$, with $c>0$ and $\lambda\geq 0$.
Theorem \ref{theorem Lamperti}, Proposition \ref{prop A-to-eps}, Proposition \ref{proposition asymptotic mean variance}, and Proposition \ref{prop recurrence3d} imply that
 
(i) if $1\leq d\leq 3$ then $(X_{d,k})_{k\geq 0}$ is neighborhood recurrent, and

(ii) if $d\geq 4$ then $(X_{d,k})_{k\geq 0}$ is transient if $\lambda<(d-3)/2$ and neighborhood recurrent if $\lambda>(d-3)/2$.

 Let
$(\mathbf{X}(t), \mathbf{V}(t))_{t\geq 0}$
  be the Markov process with generator \eqref{eq generator d} started from $(\mathbf{0},\mathbf{0})$ with gravitation $g$ and scatterer density $h(\mathbf{x})=h(x_d)=c|x_d|^\lambda$, with $c>0$ and $\lambda\geq 0$ and 
let  $(\mathbf{X}(t))_{t\geq 0}= \{(X_1(t),\dots,X_d(t))\}_{t\geq 0}$.
The process $(\mathbf{X}_k)_{k\geq 0}$ can be constructed as $(\mathbf{X}(t))_{t\geq 0}$ sampled at some random times. Hence, if $(X_{d,k})_{k\geq 0}$ visits an interval $[y, 0]$ infinitely often, so does $(X_d(t))_{t\geq 0}$. In other words, if $(X_{d,k})_{k\geq 0}$ is neighborhood recurrent then $(X_d(t))_{t\geq 0}$ is neighborhood recurrent. Next suppose that $(X_d(t))_{t\geq 0}$ is neighborhood recurrent and fix any $y<0$. It is easy to see that $(X_d(t))_{t\geq 0}$ will visit $(2y,y)$ infinitely often, a.s. The random flight construction shows that there exists $p>0$, depending on $y$, such that if $X_d(0)\in(2y,y)$ then with probability greater than $p$ there will be a scattering event at a location such that  $X_d(t)\in (y,y/2)$ before $X_d$ hits $3y$.  A standard argument based on the strong Markov property then shows that there will be infinitely many scattering events with $X_d(t)\in(y,y/2)$, a.s. It follows that $(X_{d,k})_{k\geq 0}$ is neighborhood recurrent.
We conclude that  $(X_{d,k})_{k\geq 0}$ is neighborhood recurrent if and only if $(X_d(t))_{t\geq 0}$ is neighborhood recurrent. 

It remains to show that $(X_d(t))_{t\geq 0}$ is recurrent only in the case $d=1$. It is easy to see, using continuity of $(X_d(t))_{t\geq 0}$, that neighborhood recurrence implies that all $y<0$ are visited infinitely often, a.s. 
The energy of the particle is preserved forever, so if $(\mathbf{X}(0), \mathbf{V}(0))=(\mathbf{0},\mathbf{0})$ then we may have $X_d(t_1) = 0$ for some $t_1$ only if $\mathbf{V}(t_1)= \mathbf{0}$. But if $d\geq 2$ then after every scattering event, the first coordinate of  $\mathbf{V}$ is a non-zero constant until the next scattering event, a.s. This shows that $X_d(t)\ne 0$ for all $t>0$, a.s.

If $d=1$ and $(X_d(t))_{t\geq 0}$ is neighborhood recurrent then the process will visit some interval $[y,0]$ infinitely often, a.s., and, because of the claim (i) for $(X_{d,k})_{k\geq 0}$, it will scatter within this interval. After the scattering event, it will travel upwards with probability 1/2 and reach 0 with probability $p_1>0$, depending on $y$. A standard argument based on the strong Markov property shows that  $(X_d(t))_{t\geq 0}$ will hit 0 infinitely often, a.s.
\end{proof}

\begin{proof}[Proof of Theorem \ref{theorem intro sk inv}]
This follows  from \ref{theorem Lamperti} and Proposition \ref{proposition asymptotic mean variance}.
\end{proof}

We now turn to the proof of Theorem \ref{theorem intro natural}.  The idea is to augment the result of Theorem \ref{theorem intro sk inv} to include information on the time between reflections and then make a time change argument using the continuity properties of the Skorokhod topology.

To simplify notation, let 
\begin{equation}\label{eq zldef} (Z^\lambda_t, t\geq 0) =_d \left(-\rho_{d'} \left(\frac{2}{dc^2}\left(1+\lambda \right)^2 t\right)^{1/(1+\lambda)}, t\geq 0\right)\end{equation}
where $ d' =  (d+1+2\lambda)/(2+2\lambda)$ and $\left(\rho_{d'} (t), t\geq 0\right)$ is a $d'$-dimensional Bessel process.  Note that $Z^\lambda$ is a Feller process and a straightforward calculation shows that its generator acts on $f\in C^2(-\infty,0)$ with compact support by
\begin{equation}\label{eq lambda gen} \cA^\lambda f(y) = \frac{2}{dc^2} |y|^{-2\lambda} \left[ \frac{1}{2} f''(y) - \left(\frac{d-1-2\lambda}{4|y|} \right) f'(y)\right].\end{equation}

\begin{theorem}\label{de17.5}
Consider the Markov chain $((Y_m,\Delta_m), m \geq 0)$ started from $(0,0)$ with transition operator
\[\wt \oper f(y,z) = \E\left[ f\left(y+U_d\v N(y,\bU) - \frac{g}{2}N(y,\bU)^2, \  N(y,\bU) \right)\right].\]
Fix $\eps >0$ and for $m\in \N$, let $T^{n,\eps}_m = \sum_{j=1}^m \Delta_j\cf_{\{Y_{j-1} \leq -\eps n^{1/(2+2\lambda)}\}} $.  We extend $T^{n,\eps}$ to $\R_+$ by linear interpolation.  We have the joint convergence in distribution
\[ \left( \left( n^{-1/(2+2\lambda)} Y_{[s n]}, n^{-\frac{3+2\lambda}{4+4\lambda}} T^{n,\eps}_{nt} \right) , s,t\geq 0\right) \to_d \left( \left( Z^\lambda_s, \Phi_\eps( Z^\lambda)_t\right) , s,t\geq 0\right)\]
in $D(\R_+,\R)\times D(\R_+,\R)$, where $\Phi_\eps :D(\R_+,\R)\to D(\R_+,\R)$ is defined by
\begin{align}\label{j20.4}
\Phi_\eps(f)_t = \int_0^t \frac{ \cf(f(s) \leq -\eps)}{c \sqrt{2g}\, |f(s)|^{\lambda+1/2}} ds.
\end{align}
\end{theorem}

 \begin{proof}  

Note that the map $\Phi_\eps$ is continuous in the Skorokhod topology at all continuous functions $f$ such that $\Leb(\{s : f(s)=-\eps\})=0$, where $\Leb$ stands for Lebesgue measure.  In particular, it is almost surely continuous at $( Z^\lambda_t, \ t\geq 0)$.  Hence, we conclude from  
Theorem \ref{theorem intro sk inv} that for every $\eps >0$ we have the joint convergence in distribution in $D(\R_+,\R)\times D(\R_+,\R)$,
\begin{align}\label{de17.1}
\left( \left( n^{-1/(2+2\lambda)} Y_{[s n]}, \Phi_\eps \left(n^{-1/(2+2\lambda)} Y_{[t n]} \right) \right) , s,t\geq 0\right) \to_d \left( \left( Z^\lambda_s, \Phi_\eps( Z^\lambda)_t\right) , s,t\geq 0\right).
\end{align}

Let $\cF_m = \sigma( (Y_j, \Delta_j), 0\leq j\leq m)$ and consider the martingale with respect to the filtration $(\cF_m)_{m\geq 0}$ given by
\[ W_m:= \sum_{j=1}^m \left(\Delta_j - \E\left[\Delta_j \mid \cF_{j-1}\right]\right),\ m\geq 0.\]
Define $\phi(y) =\E\left(N(y,\bU)\right)$. By the Markov property we see that $ \E\left[\Delta_j \mid \cF_{j-1}\right] = \phi\left(Y_{j-1}\right)$.

By Lemma \ref{lemma large y bound} we see that $\sup_{y}\phi(y)<\infty$ and 
\[ \xi:=\sup_{m} \E\left[\left(\Delta_m- \E\left[\Delta_m \mid \cF_{m-1}\right]\right)^2\right] <\infty.\]
By Chebyshev's and Doob's maximal inequalities we see that for every $\eps>0$ and integer $k\geq 1$,
\begin{align*}
 \P\left( \sup_{1\leq m \leq kn} \left| W_m \right| > \eps n^{\frac{3+2\lambda}{4+4\lambda}}\right) 
&\leq  \frac{1}{\eps^2 n^{(3+2\lambda)/(2+2\lambda)} } 
\E\left[\left(
\sup_{1\leq m \leq kn} | W_m| \right)^2\right] 
\leq  \frac{4}{\eps^2 n^{(3+2\lambda)/(2+2\lambda)} } \E\left[|W_{kn}|^2\right] \\
&\leq \frac{4k\xi }{\eps^2 n^{1/(2+2\lambda)}},
\end{align*}
from which it follows that $\sup_{1\leq m \leq kn} \left|n^{-\frac{3+2\lambda}{4+4\lambda}} W_m \right| $ converges to $0$ in probability as $n\to\infty$.  Similarly, if for $\eps >0$ we define
\[ W^{n,\eps}_m =   \sum_{j=1}^m \left(\Delta_j - \E\left[\Delta_j \mid \cF_{j-1}\right]\right)\cf_{\{Y_{j-1}\leq -\eps n^{1/(2+2\lambda)}\}}
=   \sum_{j=1}^m \left(\Delta_j - \phi\left(Y_{j-1}\right)\right)\cf_{\{Y_{j-1}\leq -\eps n^{1/(2+2\lambda)}\}}
,\ m\geq 0,\]
we find that  $\sup_{1\leq m \leq kn} \left|n^{-\frac{3+2\lambda}{4+4\lambda}} W^{n,\eps}_m \right| $ converges to $0$ in probability as $n\to\infty$. We record this for future reference as
\begin{align}\label{de17.2}
\sup_{1\leq m \leq kn} \left|n^{-\frac{3+2\lambda}{4+4\lambda}} \sum_{j=1}^m \left(\Delta_j - \phi\left(Y_{j-1}\right)\right)\cf_{\{Y_{j-1}\leq -\eps n^{1/(2+2\lambda)}\}} \right| \to 0,
\end{align}
in probability as $n\to\infty$.

Lemma \ref{lemma moments monomial} implies that for every $\eps >0$,
\begin{align*}
\limsup_{n\to\infty}\ &  n^{\frac{1+2\lambda}{4+4\lambda}} \sup_{y\leq -\eps n^{1/(2+2\lambda)}}\left| \phi(y) - \frac 1{c\sqrt{2g} |y|^{\lambda+1/2}} \right| \\
&=
\limsup_{n\to\infty}  \eps^{-\lambda -1/2}
\inf_{z\leq -\eps n^{1/(2+2\lambda)}} |z|^{\lambda+1/2}
\sup_{y\leq -\eps n^{1/(2+2\lambda)}}\left| \phi(y) - \frac 1{c\sqrt{2g} |y|^{\lambda+1/2}} \right| \\
&\leq
\limsup_{n\to\infty}  \eps^{-\lambda -1/2}
 \frac 1{c\sqrt{2g}}
\sup_{y\leq -\eps n^{1/(2+2\lambda)}}
c\sqrt{2g} |y|^{\lambda+1/2} \left| \phi(y) - \frac 1{c\sqrt{2g} |y|^{\lambda+1/2}} \right| \\
&=
\limsup_{n\to\infty}  \eps^{-\lambda -1/2}
 \frac 1{c\sqrt{2g}}
\sup_{y\leq -\eps n^{1/(2+2\lambda)}}
\left|c\sqrt{2g} |y|^{\lambda+1/2} \phi(y) - 1\right| =0 .
\end{align*}
This implies that for every integer $k\geq 1$, a.s.,
\begin{align}\label{de17.3}
 \limsup_{n\to\infty}& \sup_{1 \leq m \leq kn}  n^{-\frac{3+2\lambda}{4+4\lambda}} \sum_{j=1}^m \left|  \phi\left(Y_{j-1}\right) - \frac{1}{\sqrt{2g\left|Y_{j-1}\right|}\, h\left(Y_{j-1}\right)} \right|\cf_{\{Y_{j-1}\leq -\eps n^{1/(2+2\lambda)}\}}\\
&\leq 
 \limsup_{n\to\infty}  n^{-\frac{3+2\lambda}{4+4\lambda}} k n \sup_{y\leq -\eps n^{1/(2+2\lambda)}}\left| \phi(y) - \frac 1{c\sqrt{2g} |y|^{\lambda+1/2}} \right|=0.\nonumber
\end{align}
Note that,
\[\begin{split} \Phi_\eps \left(n^{-1/(2+2\lambda)} Y_{[\,\cdot\, n]} \right)_{m/n} & =  \frac{1}{n} \sum_{j=1}^m\frac{\cf_{\{n^{-1/(2+2\lambda)}Y_{j-1}\leq -\eps \}}}{\sqrt{2g\left|n^{-1/(2+2\lambda)}Y_{j-1}\right|}\, h\left(n^{-1/(2+2\lambda)}Y_{j-1}\right)} \\
& =n^{-\frac{3+2\lambda}{4+4\lambda}} \sum_{j=1}^m\frac{\cf_{\{Y_{j-1}\leq -\eps n^{1/(2+2\lambda)}\}}}{\sqrt{2g\left|Y_{j-1}\right|}\, h\left(Y_{j-1}\right)}
.\end{split}\]
Hence, 
\begin{align*}
&\sup_{1 \leq m \leq kn} \left| \Phi_\eps \left(n^{-1/(2+2\lambda)} Y_{[\,\cdot\, n]} \right)_{m/n} - n^{-\frac{3+2\lambda}{4+4\lambda}} T^{n,\eps}_{m} \right|\\
&= 
\sup_{1 \leq m \leq kn} \left| \Phi_\eps \left(n^{-1/(2+2\lambda)} Y_{[\,\cdot\, n]} \right)_{m/n} -  n^{-\frac{3+2\lambda}{4+4\lambda}} \sum_{j=1}^{m}\cf_{\{Y_{j-1}\leq -\eps n^{1/(2+2\lambda)}\}}\Delta_{j} \right|\\
&= \sup_{1 \leq m \leq kn} \left| n^{-\frac{3+2\lambda}{4+4\lambda}} \sum_{j=1}^m\frac{\cf_{\{Y_{j-1}\leq -\eps n^{1/(2+2\lambda)}\}}}{\sqrt{2g\left|Y_{j-1}\right|}\, h\left(Y_{j-1}\right)} -  n^{-\frac{3+2\lambda}{4+4\lambda}} \sum_{j=1}^{m}\cf_{\{Y_{j-1}\leq -\eps n^{1/(2+2\lambda)}\}}\Delta_{j} \right|\\
&\leq
\sup_{1\leq m \leq kn} \left|n^{-\frac{3+2\lambda}{4+4\lambda}} \sum_{j=1}^m \left(\Delta_j - \phi\left(Y_{j-1}\right)\right)\cf_{\{Y_{j-1}\leq -\eps n^{1/(2+2\lambda)}\}} \right|\\
& \quad +
\sup_{1 \leq m \leq kn}  n^{-\frac{3+2\lambda}{4+4\lambda}} \sum_{j=1}^m \left|  \phi\left(Y_{j-1}\right) - \frac{1}{\sqrt{2g\left|Y_{j-1}\right|}\, h\left(Y_{j-1}\right)} \right|\cf_{\{Y_{j-1}\leq -\eps n^{1/(2+2\lambda)}\}}.
\end{align*}
This, \eqref{de17.2} and \eqref{de17.3} imply that for fixed $\eps>0$ and $k$,
\begin{align*}
&\sup_{1 \leq m \leq kn} \left| \Phi_\eps \left(n^{-1/(2+2\lambda)} Y_{[\,\cdot\, n]} \right)_{m/n} -  n^{-\frac{3+2\lambda}{4+4\lambda}} T^{n,\eps}_{m} \right|
\to 0,
\end{align*}
in probability, as $n\to \infty$.
It follows from this and \eqref{de17.1} that for every $\eps >0$ we have the joint convergence in distribution
\[ \left( \left( n^{-1/(2+2\lambda)} Y_{[s n]}, n^{-\frac{3+2\lambda}{4+4\lambda}} T^{n,\eps}_{nt}  \right) , s,t\geq 0\right) \to_d \left( \left( Z^\lambda_s, \Phi_\eps( Z^\lambda)_t\right) , s,t\geq 0\right)\]
in $D(\R_+,\R)\times D(\R_+,\R)$, and the result follows.  
\end{proof}

\begin{remark}
We conjecture that the convergence in Theorem \ref{de17.5} can be extended to include the case $\eps =0$.  One reason to believe this is that the limiting process is still well defined.  From the basic properties of Bessel processes it follows that for every fixed $t^*\geq 0$ we have $\lim_{\eps \to 0} m(\{s \leq t^* :  Z^\lambda_s \geq-\eps\})=0$ almost surely.  Consequently, we have that
\begin{align}\label{j20.3}
\lim_{\eps\to 0} \Phi_\eps( Z^\lambda) = \left(\int_0^t \frac{ 1}{c \sqrt{2g} | Z^\lambda_s|^{\lambda+1/2}} ds, t\geq 0 \right) \equiv \Phi( Z^\lambda), \quad \text{a.s.}
\end{align}
A standard occupation density computation for Bessel processes shows that $\Phi( Z^\lambda)_t< \infty$ a.s., for every $t\geq 0$.  The problem comes in controlling the amount of time spent between collisions when $Y$ is near $0$, which contribute constant order time.  We note that the same difficulty arises in the periodic Galton Board model, studied in \cite{CD09}, where the authors avoided this complication by assuming the particle had a sufficiently large initial velocity and was reflected down at the corresponding level.  In \cite{RaviT} the authors considered a model similar to ours when $h\equiv 1$ and, in that setting, were able to overcome this difficulty through different methods.
\end{remark}

Theorem \ref{de17.5} allows us to obtain a scaling limit for the continuous time particle path (away from $0$). In addition to keeping track of time we need to keep track of the direction of reflection. That is, we consider the Markov chain $((Y_m,\Delta_m,\bU^m), m\geq 0)$ with transition operator
\[ \wh \oper f(y,z,w) = \E\left[ f\left(y+U_d\v N(y,\bU) - \frac{g}{2}N(y,\bU)^2, \  N(y,\bU), \bU \right)\right],\]
started from $(0,0,(0,\dots,0,-1))$. 
Let $T_m = \sum_{j=0}^m \Delta_j $. The $d$-th component of the path of the particle is then given by
\begin{equation} \label{eq particle2}
Y(t) = Y_{m-1}+U^m_d\sqrt{2g|Y_{m-1}|}(t-T_{m-1}) - \frac{g}{2}(t-T_{m-1})^2 \quad \textrm{on} \quad T_{m-1} \leq t < T_m, m\geq 1. 
\end{equation}

The following lemma is likely to be known but we could not find a reference. 

Let $\R^* = \R \cup \{\infty\}$ and $\R^*_+ = \R_+ \cup \{\infty\}$. By convention, $\inf \emptyset = \infty$ and for any function $f$, $f(\infty) = \infty$.
For $f\in D(\R_+,\R_+)$, define $\Psi:D(\R_+,\R_+) \to D(\R_+,\R^*_+)$ by $\Psi(f)(t) = \inf\left\{s : f(s)>t\right\}$.

\begin{lemma}\label{lemma e-inverse}
If $f \in D(\R_+,\R_+)$ is continuous and strictly increasing with $\lim_{t\to\infty} f(t)=\infty$, then $\Psi(f) \in D(\R_+,\R_+)$ and $\Psi$ is continuous at $f$.
\end{lemma}

\begin{proof}
First we prove that for any $h\in D(\R_+,\R_+)$, the function $\Psi(h)$ is in $D(\R_+,\R^*_+)$.  It is clear that $\Psi(h)$
is a non-decreasing function. 
Since the function $\Psi(h)$ is monotone, it has left and right limits at every point. It remains to show that it is right-continuous. Since $\Psi(h)$ is non-decreasing, we have $\lim_{s\downarrow t}  \Psi(h)(s) \geq \Psi(h)(t)$ for every $t$.
Consider any $t $ and an arbitrarily small $\delta>0$, and let $b = \Psi(h)(t)$. 
If $h(b) \leq t$ then there must exist $b_1\in(b,b+\delta)$ and $t_1 >t$ such that $h(b_1) = t_1$. This claim holds also in the case $h(b) > t$, by the right-continuity of $h$. For all $s\in(t, t_1)$ we have 
$\Psi(h)(s) \leq b_1 < b+ \delta$. Since $\delta>0$ is arbitrarily small, this implies that $\lim_{s\downarrow t}  \Psi(f)(s) \leq \Psi(h)(t)$. In view of the previously proved opposite inequality, we conclude that $\Psi(h)$ is right continuous at $t$. This completes the proof that $\Psi(h) \in D(\R_+,\R^*_+)$.

Now suppose that $f$ satisfies the hypotheses of the lemma and that $f_n \in D(\R_+,\R_+)$ is a sequence converging to $f$. Since $f$ is continuous and strictly increasing, the function $\Psi(f)$ is also continuous and strictly increasing.
Fix any $T < \infty$. It suffices to show that
\begin{align*}
\lim_{n\to\infty} \sup_{t\in[0,T]}  |\Psi(f_n)(t) -\Psi(f)(t)| =0.
\end{align*}
Suppose otherwise. Then there exist $\eps >0$, a subsequence $n_k$ and a sequence $t_{n_k}$ of points in $[0,T]$, such that $|\Psi(f_{n_k})(t_{n_k}) -\Psi(f)(t_{n_k})| > \eps$ for all $k$. By compactness, we may suppose that $t_{n_k} \to t_\infty \in [0,T]$ as $k\to \infty$. 
We will assume that $t_\infty \in(0,T)$. The argument requires only small modifications when $t_\infty$ is 0 or $T$. 

Let $s_\infty = \Psi(f)(t_\infty)$ and
\begin{align*}
\delta =  \min(f(s_\infty- \eps/4) - f(s_\infty- \eps/2), f(s_\infty+ \eps/2) - f(s_\infty+ \eps/4)).
\end{align*}

Since $f(s_\infty) = t_\infty$, $f$ is strictly increasing and $t_{n_k} \to t_\infty$, there exists $k_1$ such that for all $k\geq k_1$,
\begin{align}\label{de15.2}
f(s_\infty- \eps/4) \leq t_{n_k} \leq
f(s_\infty+ \eps/4).
\end{align}

Since $f$ is continuous, $f_n\to f$ uniformly on compact sets. Let $k_2\geq k_1$ be so large that for $k\geq k_2$.
\begin{align}\label{de15.3}
|\Psi(f)(t_\infty) - \Psi(f)(t_{n_k})|  &< \eps/4,\\
 \sup_{t\in[0, s_\infty - \eps/2]} |f_{n_k}(t) -f(t)| &< \delta/4,\label{de15.4}\\
 \sup_{t\in[ s_\infty + \eps/2,T]} |f_{n_k}(t) -f(t)| &< \delta/4.\nonumber
\end{align}
It follows from the definition of $\delta$ and \eqref{de15.4}  that
\begin{align*}
 \sup_{t\in[0, s_\infty - \eps/2]} f_{n_k}(t)  &< f(s_\infty- \eps/4).
\end{align*}
This, \eqref{de15.2}, the definition of $s_\infty$ and \eqref{de15.3} imply that
\begin{align}\label{de15.5}
\Psi(f_{n_k})(t_{n_k}) \geq
\Psi(f_{n_k})(f(s_\infty- \eps/4)) \geq s_\infty- \eps/2
= \Psi(f)(t_\infty) - \eps/2 \geq
\Psi(f)(t_{n_k}) - 3\eps/4.
\end{align}
The following estimates can be obtained in an analogous way,
\begin{align*}
\Psi(f_{n_k})(t_{n_k}) \leq
\Psi(f_{n_k})(f(s_\infty+ \eps/4)) \leq s_\infty+ \eps/2
= \Psi(f)(t_\infty) + \eps/2 \leq
\Psi(f)(t_{n_k}) + 3\eps/4.
\end{align*}
We combine this with \eqref{de15.5} to obtain $|\Psi(f_{n_k})(t_{n_k}) -\Psi(f)(t_{n_k})| \leq 3\eps/4$. This contradicts the definition of the sequence $t_{n_k}$. This contradiction completes the proof.
\end{proof}

In order to apply this lemma, we need the following proposition.
Let $\tau_{v+} = \inf \{t: Z^\lambda(t) > v\}$.

\begin{proposition}
For all $y<v<0$,
\[ \P_y\left(\lim_{t\to\infty}   \int_0^t \frac{1}{c\sqrt{2g} \,|Z^\lambda(s \wedge \tau_{v+})|^{1/2+\lambda}} ds = \infty\right) =1.\]
\end{proposition}

\begin{proof}
The result is trivial on the set where $(Z^\lambda(t \wedge \tau_{v+}), t\geq 0)$ is absorbed at $v$.  It follows from \eqref{eq lambda gen} that the scale function $G$ and speed measure $m$ for $Z^\lambda$ are given by
\[ G(y) = \int_{-1}^y |u|^{\lambda - (d-1)/2 } du 
\quad \textrm{and}\quad  m(dy) = 
\frac{dc^2}2 |y|^{\lambda+(d-1)/2} dy.\] 
If $G(-\infty)=-\infty$, then  $(Z^\lambda(t \wedge \tau_{v+}), t\geq 0)$ is absorbed at $v$ with probability $1$, so we may assume that $G(-\infty)$ is finite.  Note that this implies that $d>3$.  In this case there exists $C>0$ such that for all $y<-1$
\[ (G(y) - G(-\infty)) \left(\frac{1}{c\sqrt{2g}\, |y|^{1/2+\lambda}}\right) \frac{dm}{dy}(y) \geq C\sqrt{|y|}.\]
Since $\int_{-\infty}^y \sqrt{|u|}du = \infty$ for all $y\in \R$, the result is an application of \cite[Theorem 2.11]{MU12}. 
\end{proof}

\begin{theorem}\label{theorem monomial path convergence}
Fix $y<v <0$ and define $\tau^n_{y-} = \inf\{m : Y_m\leq n^{1/(2+2\lambda)}y\}$ and $\tau^n_{v+} = \inf\{m > \tau^n_{y-} : Y_m\geq n^{1/(2+2\lambda)}v\}$.  For $(Y(t),t\geq 0)$ as defined in \eqref{eq particle2} and $y<v<0$ we have the following convergence in distribution on $D(\R_+,\R)$,
\[ \left(n^{-\frac{1}{2+2\lambda}}Y\left(\left(n^{ \frac{3+2\lambda}{4+4\lambda}}t+T_{\tau^n_{y-}} \right)\wedge  T_{\tau^n_{v+}}\right),\ t \geq 0\right) \rightarrow ( Z^\lambda (A(t)\wedge \tau_{v+}), \ t\geq 0),\]
where $ Z^\lambda$ is the diffusion \eqref{eq zldef} started from $y$ and 
\[ A(t) = \Psi\left( \Phi\left( Z^\lambda(\,\cdot\, \wedge \tau_{v+})\right)\right).\]
\end{theorem}

\begin{remark}\label{j20.9}
The theorem remains true replacing $A$ with $\Psi\left( \Phi_\eps( Z^\lambda(\,\cdot\, \wedge \tau_{v+}))\right)$ for any $0<\eps<|v|$, where $\Phi_\eps$ is defined in \eqref{j20.4}.  
\end{remark}

\begin{proof}[Proof of Theorem \ref{theorem monomial path convergence}]
Fix $0<\eps<|v|$ and define $\lambda' = (3+2\lambda)/(4+4\lambda)$.  Recall
$T^{n,\eps}_m$ from Theorem \ref{de17.5} and let
\[\begin{split} A_n(t)  & =\Psi\left(n^{-\lambda'} \left(T_{(\tau^n_{y-}+n
\,\cdot\,)\wedge  \tau^n_{v+}}-T_{ \tau^n_{y-}}\right)+
\left(\,\cdot\,-n^{-1}\left(\tau^n_{v+}-\tau^n_{y-}\right)\right)^+ 
\frac{1}{\sqrt{2g| v|} \, h( v)} \right)(t) \\
&=  \Psi\left(n^{-\lambda'} \left(T^{n,\eps}_{(\tau^n_{y-}+n\,\cdot\,)\wedge  \tau^n_{v+}}-T^{n,\eps}_{ \tau^n_{y-}}\right)+
\left(\,\cdot\,-n^{-1}\left(\tau^n_{v+}-\tau^n_{y-}\right)\right)^+ 
\frac{1}{\sqrt{2g| v|} \, h( v)} \right)(t) 
.\end{split}\]
Using Theorem \ref{de17.5}, Lemma \ref{lemma e-inverse}, and the Skorokhod-continuity of composition with a continuous function (see e.g. \cite[Section 17]{B68}), we have that
\[ \left(n^{-\frac{1}{2+2\lambda}}Y_{[nA_n(t)]\wedge  (\tau^n_{v+}-\tau^n_{y-})},\ t \geq 0\right) \rightarrow ( Z^\lambda(A(t)\wedge \tau_{v+}), \ t\geq 0).\]

Observe that for all $0 \leq m\leq \tau^n_{v+}-\tau^n_{y-}$ we have $nA_n\left(n^{-\lambda'}\left(T^{n,\eps}_{\tau^n_{y-}+m}-T^{n,\eps}_{\tau^n_{y-}}\right)\right)=m$ 
and, as a result, if $T^{n,\eps}_{\tau^n_{y-}+m-1}-T^{n,\eps}_{\tau^n_{y-}}\leq n^{\lambda'}t < T^{n,\eps}_{\tau^n_{y-}+m}-T^{n,\eps}_{\tau^n_{y-}}$ then $m-1\leq nA_n(t)< m$.  

Define $\hat T^n_m = T_{\tau^n_{y+}+m}$, fix $S>0$ and observe that
\begin{align}
& \sup_{0\leq t \leq S} 
n^{-\frac{1}{2+2\lambda}}
\left| Y\left(\left(n^{ \lambda'}t+T_{\tau^n_{y-}} \right)\wedge  T_{\tau^n_{v+}}\right) -  Y_{[nA_n(t)]\wedge(\tau^n_{v+}-\tau^n_{y-})}\right| \label{j20.6} \\
& \quad\leq \sup_{m\leq \left\lceil nA_n\left(S\wedge  n^{-\lambda'}(T_{\tau^n_{v+}}-T_{\tau^n_{y-}})\right)\right\rceil} \sup_{\hat T^n_{m-1}\leq t \leq \hat T^n_{m}}\nonumber \\
&\qquad\qquad\left\{
\left|U^{\tau^n_{y+} +m-1}_d\sqrt{\frac{2g|Y_{\tau^n_{y-}+m-1}|}{n^{1/(2+2\lambda)}}}n^{-\frac{1}{4+4\lambda}}(t-\hat T^n_{m-1}) - \frac{g}{2n^{1/(2+2\lambda)}}(t-\hat T^n_{m-1})^2 \right| \right\}\nonumber \\
&\quad\leq \sup_{m\leq \left\lceil nA_n\left(S\wedge  n^{-\lambda'}(T_{\tau^n_{v+}}-T_{\tau^n_{y-}})\right)\right\rceil} \sqrt{\frac{2g|Y_{\tau^n_{y-}+m-1}|}{n^{1/(2+2\lambda)}}}n^{-\frac{1}{4+4\lambda}}(\hat T^n_{m}- \hat T^n_{m-1}) 
\nonumber \\
 &\qquad + \sup_{m\leq \left\lceil nA_n\left(S\wedge  n^{-\lambda'}(T_{\tau^n_{v+}}-T_{\tau^n_{y-}})\right)\right\rceil}\frac{g}{2n^{1/(2+2\lambda)}}(\hat T^n_{m}- \hat T^n_{m-1})^2 . \nonumber
\end{align}

Since $Z^\lambda$ almost surely fluctuates across levels, the convergence in Theorem \ref{de17.5} occurs jointly with the hitting time of $v$, so that
\begin{equation} \label{j20.1} 
\left((n^{-\frac{1}{2+2\lambda}} Y_{[nA_n(t)]\wedge(\tau^n_{v+}-\tau^n_{y-})},\ A_n(t),\ n^{-1}(\tau^n_{v+}-\tau^n_{y-})) , t\geq 0\right) 
\overset{d}{\longrightarrow} (( Z^\lambda(A(t)\wedge \tau_{v+}), A(t), \tau_{v+}), t\geq 0).
\end{equation}
Let
\begin{align*}
B_n = \Bigg\{ &
 \sup_{m\leq \left\lceil nA_n\left(S\wedge  n^{-\lambda'}(T_{\tau^n_{v+}}-T_{\tau^n_{y-}})\right)\right\rceil} \frac{1}{n^{2+2\lambda}}|Y_{\tau^n_{y-}+m}| \leq M ,\quad A_n\left(S\wedge  n^{-\lambda'}(T_{\tau^n_{v+}}-T_{\tau^n_{y-}})\right) \leq M, \\ 
&\quad Y_{\tau^n_{v+}} <n^{1/(2+2\lambda)}( v+\delta) \Bigg\}.
\end{align*}
It follows from \eqref{j20.1} that for every $p_1<1$
there exist $M>0$ and $0<\delta < |v|$ such that for large $n$, $\P(B_n) > p_1$.
We use \eqref{j20.6} to conclude that for $\eps \in(0,1)$ there exist $C_1, C_2>0$ such that 
\begin{multline*} \P\left( \sup_{0\leq t \leq S} \left| Y\left(\left(n^{ \lambda'}t+T_{\tau^n_{y-}} \right)\wedge  T_{\tau^n_{v+}}\right) -  Y_{[nA_n(t)]\wedge(\tau^n_{v+}-\tau^n_{y-})}\right| > \eps, B_n\right)\\
 \leq C_1n\sup_{y\leq n^{1/(2+2\lambda)}( v+\delta), \bU \in \bS^{d-1}}\P(N(y,\bU) > C_2\eps). \end{multline*}
The right hand side goes to $0$ by Lemma \ref{lemma large y bound}, 
applied with a large enough value of $p$, and
using Markov's inequality.  Since $\P(B_n) \to 1$, it follows from \eqref{j20.1} that 
\[ \left(\frac{1}{n^{\frac{1}{2+2\lambda}}}Y\left(\left(n^{ \frac{3+2\lambda}{4+4\lambda}}t+T_{\tau^n_{y-}} \right)\wedge  T_{\tau^n_{v+}}\right),\ t \geq 0\right) \rightarrow ( Z^\lambda (A(t)\wedge \tau_{v+}), \ t\geq 0).\qedhere\]

\begin{proof}[Proof of Theorem \ref{theorem intro natural}]
This result is a consequence of Theorem \ref{theorem monomial path convergence} and a straightforward generator computation.
\end{proof}

\section{More general densities: Theorem \ref{THM GENERAL DENSITY}}\label{highdim}
The analysis in this case is similar to the analysis is Section \ref{highdimpower}.  Instead of appealing to Lamperti's results we appeal to classical results on convergence for martingale problems, in particular we use \cite[Theorem IX.4.21]{JS03}.  We will provide detail where needed, but in many cases we will state a result and leave it to the reader to make the necessary modifications to the corresponding proofs in Section \ref{highdimpower}.  In all such cases the modification is completely routine.  We also refer the reader to \cite{HRW15}, which establishes similar results in a more general setting.

In the setting of Theorem \ref{THM GENERAL DENSITY}, the transition operator for the Markov chain $(Y^n_{m}, m\geq 0)$ representing $x_d$-coordinates of the particle at reflection times is given by
\begin{align}\label{j14.1}
\oper^nf(y) = \E \left[ f\left(y+ U_d\sqrt{2g|y|}n^{-1/4}N^n(y, \bU) - \frac{g}{2\sqrt{n}}N^n(y, \bU) ^2\right)\right]
\end{align}
where $\bU$ is uniform on $\bS^{d-1}$ and conditionally given that $\bU=\bu$,
\begin{align}\label{de16.5}
\P(N^n(y,\bu) >t)  = \exp\Bigg[-\int_0^t &\sqrt{n} h\left(y+u_d\v n^{-1/4}s - \frac{g}{2\sqrt{n}}s^2\right)\\
 &\times \sqrt{\vs n^{-1/2}(1-u_d^2)+\left(\v u_dn^{-1/4}-\frac{g}{\sqrt{n}}s\right)^2}ds\Bigg]. \nonumber
\end{align}
To ease our notation, define
\[\alpha_n(y,\bu) = u_d\v n^{-1/4}N^n(y,\bu) - \frac{g}{2\sqrt{n}}N^n(y,\bu)^2,\]
so that $\oper^nf(y) = \E f\left(y+\alpha_n(y,\bU)\right)$.

As before, we first establish a result for the process observed at reflection times.

\begin{theorem}\label{theorem skeletal convergence d}
Consider any $-\infty<y< v <0$ and suppose that $(Y^n_m,m\geq 0)$ is the Markov chain with transition operator $\oper^n$ started from $y$. 
When $n\to \infty$, the processes $(Y^n_{[nt] \wedge \tau^n_{v+}},t \geq 0)$
converge in distribution on the Skorokhod space to a diffusion whose generator extends the operator that acts on $f\in C^2(-\infty,0)$ with compact support by  
\begin{align}\label{ja24.1}
\mathcal{\bar A}_{h,d}f(y) = \frac{1}{d h(y)^2}f''(y) - \frac{1}{dh(y)^2}\left(\frac{d-1}{2|y|} + \frac{h'(y)}{h(y)}\right)f'(y),
\end{align} 
starting at $y$ and stopped at the hitting time of $v$.
\end{theorem}

The proof of the theorem will be preceded by a number of lemmas.

\begin{proposition}\label{proposition d dimensional limit process}
Suppose that $(Z_t,t\geq 0)$ is a Feller diffusion on $(-\infty, 0)$ whose generator extends $\mathcal{\bar A}_{h,d}f(y)$.
The point $-\infty$ is inaccessible. We have $\P(\lim_{t\to\infty} Z_t=-\infty)=0$ if $d\leq 3$ while 
for every $d\geq 4$,
$\P(\lim_{t\to\infty} Z_t=-\infty)$ is equal to 0 for some $h$ and is strictly positive for some other $h$. 
\end{proposition}

\begin{proof}
The scale function $G$ and speed measure $m$ for $Z$ are given by, 
\begin{align}\label{j4.3}
G(y) = \int_{-1}^y \frac{h(u)}{|u|^{(d-1)/2}} du \quad \textrm{and}\quad  m(dy) = dh(y)|y|^{(d-1)/2} d y,
\quad \text{  for  } y\in (-\infty,0).
\end{align}
For the boundary classification we use \cite[Theorem VI.3.2]{IK81}.  Define
\[ \kappa(y) = \int_{-1}^y \left[\int_{-1}^u d|s|^{(d-1)/2}h(s) ds\right]  \frac{h(u)}{|u|^{(d-1)/2}}du.\]
Observe that $\limsup_{y\to- \infty} G(y)  < 0$. Since $\inf_{y\leq -1}h(y)>0$, it follows easily that $\lim_{y\downarrow -\infty} \kappa(y) =\infty$.  Thus $-\infty$ is inaccessible.  

Note that $\lim_{y\to-\infty} G(y)=-\infty$ if $d\leq 3$, which shows that $\P(\lim_{t\to\infty} Z_t=-\infty)=0$ in this case.  For $d\geq 4$, observe that $\lim_{y\to-\infty} G(y)>-\infty$ if $h\equiv 1$ so for this $h$ we have $\P(\lim_{t\to\infty} Z_t=-\infty)>0$.  However, if $h(y)=|y|^{\beta}$ for $\beta > (d-3)/2$, then $\lim_{y\to-\infty} G(y)=-\infty$ and, consequently,
$\P(\lim_{t\to\infty} Z_t=-\infty)=0$.
\end{proof}

The next lemma, left without proof, is an adaptation of Lemma \ref{lemma large y bound} to the present context.

\begin{lemma}\label{lemma Lp_boundd}
If $a<0$ and $1\leq p<\infty$ then family $\{ n^{1/4}N^n(y,\bu) : (y,\bu ) \in (-\infty,a] \times \bS^{d-1} \times \N\}$ is bounded in $L^p$.  
\end{lemma}

Next, we need the result corresponding to Lemma \ref{lemma large y distribution}.  Because of the new scaling we need some uniformity in the distributional convergence of Lemma \ref{lemma large y distribution} and the following lemma makes this precise.

\begin{lemma} \label{lemma intermediate moments}
If $a<0$, $1\leq p<\infty$, and $r>0$ then
\[ \lim_{n\to\infty} \sup_{(y,\bu)\in (-\infty,a]\times \bS^{d-1} } \left| \E \left( \left[n^{p/4}N^n(y,\bu)^p\right] \wedge r^p \right) - \int_0^r pt^{p-1} \exp\left(- h(y) \v t\right) \ dt \right| =0.\]
\end{lemma}

\begin{proof}
Observe that
\[\begin{split}  \E \left( \left[n^{p/4}N^n(y,\bu)^p\right] \wedge r^p \right) & =   p\int_0^\infty t^{p-1} \P\left(\left[n^{1/4}N^n(y,\bu)\right]\wedge r>t\right) dt \\
&=  p\int_0^r t^{p-1} \P\left(N^n(y,\bu)>n^{-1/4}t\right) dt
.\end{split}\]
Using \eqref{de16.5} and the change of variables $v=n^{1/4}s$ we compute
\begin{align}  \label{j4.1}
\P&\left(N^n(y,\bu)>n^{-1/4}t\right)  \\
&= \exp\left[-\int_0^t \v h\left(y+u_d\v n^{-1/2}v - \frac{g}{2n}v^2\right)
 \sqrt{ (1-u_d^2)^2+\left(u_d-\frac{g}{\sqrt{2g|y|n}}v\right)^2}dv\right].
\nonumber
\end{align} 
If $A\subset (-\infty,0)$ is compact then the function
\[\alpha(y,\bu, s) := h\left(y+u_d\v s - \frac{g}{2}s^2\right) \sqrt{\vs  (1-u_d^2)^2+\left(\v u_d-gs\right)^2}\]
is continuous and, therefore, uniformly continuous on the compact set 
$A\times \bS^{d-1}\times [0,r]$.
It follows that the functions
\[\alpha(y,\bu, s n^{-1/2}) = h\left(y+u_d\v n^{-1/2}s - \frac{g}{2n}s^2\right) \sqrt{\vs  (1-u_d^2)^2+\left(\v u_d-\frac{g}{\sqrt{n}}s\right)^2}\]
converge uniformly to $ h(y)\v$ on $A\times \bS^{d-1}\times [0,r]$ as $n\to\infty$.  Consequently $\P\left(N^n(y,\bu)>n^{-1/4}t\right)$ converges uniformly to  $\exp\left(- h(y) \v t\right)$ on this set, so that
\begin{align}\label{j4.2}
\lim_{n\to\infty} \sup_{(y,\bu)\in A \times \bS^{d-1}} \left| \E \left( \left[n^{p/4}N^n(y,\bu)^p\right] \wedge r^p \right) - \int_0^r pt^{p-1} \exp\left(- h(y) \v t\right) \ dt \right| =0.
\end{align}

Since $h$ is bounded away from  $0$ for large $|y|$, there is some $\delta>0$, independent of $n \in \N$ and $s\in [0,r]$, such that
\[\liminf_{y\to-\infty} h\left(y+u_d\v n^{-1/2}s - \frac{g}{2n}s^2\right) \sqrt{ (1-u_d^2)^2+\left(u_d-\frac{g}{\sqrt{2g|y|n}}s\right)^2} >\delta. \]
This and \eqref{j4.1} imply that
\[ \lim_{y\to -\infty} \sup_{n\in \N} 
\sup_{\bu\in  \bS^{d-1}}  \E \left( \left[n^{p/4}N^n(y,\bu)^p\right] \wedge r^p \right) =0.\]
We use the fact that $h$ is bounded away from  $0$ for large $|y|$ again to see that,
\[\lim_{y\to-\infty}  \int_0^r pt^{p-1} \exp\left(- h(y) \v t\right) \ dt =0.\]
Thus, given $\eps <0$, we can find $b<a$ such that 
\[\sup_{y\leq b} \sup_{n\in \N}   \E \left( \left[n^{p/4}N^n(y,\bu)^p\right] \wedge r^p \right) + \sup_{y\leq b} \int_0^r pt^{p-1} \exp\left(- h(y) \v t\right) \ dt  < \eps.\]
This and an application of \eqref{j4.2} with $A = [b,a]$ prove the lemma.
\end{proof}

Using this, one argues as in the proof of Lemma \ref{lemma moments monomial} to show the following result.

\begin{lemma}\label{lemma momentsd}
If $a<0$ and $1\leq p<\infty$ then
\[ \lim_{n\to\infty} \sup_{(y,\bu)\in (-\infty,a]\times\bS^{d-1}} \left| \E \left(n^{p/4}N^n(y,\bu)^p\right) - \int_0^\infty pt^{p-1} \exp\left(- h(y) \v t\right) \ dt \right| =0.\]
\end{lemma}

Using these bounds, one proves the following result using the same method as in the proof of Lemma \ref{de16.6}.

\begin{lemma}\label{lemma higher dim fluctuations}
If $A\subseteq (-\infty,0)$ is compact then
\[ \lim_{n\to\infty}\sup_{y\in A} \left| n^{3/4} \E\left[ U_dN^n(y,\bU) \right] -  \frac{g  \left( h(y)- 2|y| h'(y)  \right)}{dh(y)^3(\vs)^{3/2} }\right|=0.\]
\end{lemma}

  For $n\geq 1$, let
\begin{align*}
 K^n(y,\,\cdot\,) &= n\P\left( \alpha_n(y, \bU) \in \,\cdot\,\right),\\
 b^n(y) &= \int_\R z K^n(y,dz) =  n\E\left[ \alpha_n(y, \bU)  \right] ,\\
c^n(y) &= \int_\R z^2 K^n(y,dz) = n\E\left[ \alpha_n(y, \bU)^2  \right].
\end{align*}

\begin{proposition}\label{prop JScond}
Let $A\subset (-\infty,0)$ be compact.
\begin{enumerate}
\item $\displaystyle \lim_{n\to\infty} \sup_{y\in A} \left\|b^n(y) +\frac{1}{dh(y)^2}\left(\frac{d-1}{2|y|} + \frac{h'(y)}{h(y)}\right)\right\|=0$,
\item $\displaystyle \lim_{n\to\infty} \sup_{y\in A} \left\| c^n(y) - \frac{2}{dh(y)^2}\right\|=0$,
\item For every $\rho>0$, $\displaystyle \lim_{n\to\infty} \sup_{y\in A} \int_\R z^2\cf_{(|z|>\rho)} K^n(y, dz) =0$.
\end{enumerate}
\end{proposition}

\begin{proof}
Using Lemmas \ref{lemma momentsd} and \ref{lemma higher dim fluctuations}, we see that, uniformly on $A$,
\[\begin{split} b^n(y)  =  n\E\left[ \alpha_n(y, \bU)  \right] &= n\E\left[ U_d\v n^{-1/4}N^n(y,\bU) - \frac{g}{2\sqrt{n}}N^n(y,\bU)^2\right]\\
& = \v n^{3/4}\E\left[ U_d N^n(y,\bU)\right] -  \frac{g}{2} n^{1/2}\E\left[N^n(y,\bU)^2\right] \\
& \rightarrow  \frac{g  \left( h(y)- 2|y| h'(y)  \right)}{dh(y)^3 \vs } - \frac{g}{2} \frac{2}{h(y)^2 \vs} \\
& =-\frac{1}{dh(y)^2}\left(\frac{d-1}{2|y|} + \frac{h'(y)}{h(y)}\right)
.\end{split}\]
Similarly, Lemma \ref{lemma momentsd} implies that uniformly on $A$,
\[\begin{split}c^n(y) & = n\E\left[ \alpha_n(y, \bU)^2  \right] \\
&= \vs n^{1/2} \E\left[ U^2_d N^n(y,\bU)^2 \right] - gn^{1/4}\E\left[ U_d\v N^n(y,\bU)^3 \right] + \E\left[\frac{g^2}{2}N^n(y,\bU)^4 \right]\\
& \rightarrow  \frac{2}{d h(y)^2}
.\end{split} \]
The third assertion of the proposition is a consequence of Lemma \ref{lemma Lp_boundd} and Markov's inequality.
\end{proof}

\begin{proof}[Proof of Theorem \ref{theorem skeletal convergence d}]
Proposition \ref{prop JScond} shows that the hypotheses of \cite[Theorem IX.4.21]{JS03} are satisfied for a version of the chain with appropriate cutoffs at $u,v$ with $u<y<v<0$.  As shown in Proposition \ref{proposition d dimensional limit process}, $-\infty$ is inaccessible for the limiting diffusion, which allows the cutoff at $u$ to be removed.
\end{proof}

In order to go back to the true process from the process observed at reflection times we need the following result about the limiting time change.

\begin{proposition}\label{j4.4}
Suppose that $(Z_t,t\geq 0)$ is a Feller diffusion on $(-\infty, 0)$ whose generator extends $\mathcal{\bar A}_{h,d}f(y)$ defined in \eqref{ja24.1}.
For all $d$ and for all $h$ satisfying our hypotheses we have that for all $y<v<0$,
\[ \P_y\left(\lim_{t\to\infty}   \int_0^t \frac{1}{\sqrt{2g|Z(s \wedge \tau_{v+})|} \, h(Z(s \wedge \tau_{v+}))} ds = \infty\right) =1.\]
\end{proposition}

\begin{proof}
Recall the scale function and speed measure given in  \eqref{j4.3}.
The result is trivial on the set where $(Z(t \wedge \tau_{v+}), t\geq 0)$ is absorbed at $v$.  If $G(-\infty)=-\infty$, then  $(Z(t \wedge \tau_{v+}), t\geq 0)$ is absorbed at $v$ with probability $1$, so we may assume that $G(-\infty)$ is finite.  Note that this implies that $d>3$.  In this case there exists $C>0$ such that for all $y<-1$
\[ (G(y) - G(-\infty)) \left(\frac{1}{\sqrt{2g|y|}\,h(y)}\right) \frac{dm}{dy}(y) \geq C\sqrt{|y|}.\]
Since $\int_{-\infty}^y \sqrt{|u|}du = \infty$ for all $y\in \R$, the result is an  application of \cite[Theorem 2.11]{MU12}. 
\end{proof}

\begin{proof}[Proof of Theorem \ref{THM GENERAL DENSITY}]
Combining Theorem \ref{theorem skeletal convergence d} with Proposition 
\ref{j4.4}, the proof consists of making straightforward modifications to the time change argument in the proof of Theorem \ref{theorem monomial path convergence}.
\end{proof}

\appendix

\section{Reflection direction}\label{refdir}

This short section presents an elementary fact about the classical (specular) reflection. The claim is known in dimension $d=3$ (see, for example, the discussion of  the so-called hard-sphere scattering in \cite[Sect.~4.8]{Mechanics}) but we could not find a reference for the analogous result in all dimensions $d\geq 2$. 

Suppose that $d\geq 2$.
Let $\bS^{d-1}$ be the unit sphere in $\R^d$ and let $\be_1,\dots,\be_d$ be the standard basis for $\R^d$. Let $\bB^{d-1} = \{(0,x_2,\dots, x_d) \in \R^d: x_2^2 + \dots + x_d^2  \leq 1\}$. Let $\bb$ be a random vector with the uniform distribution in $\bB^{d-1}$ and let $\calL$ be the random straight line $\{\bb + a \be_1, a\in \R\}$. Suppose that a light ray starts from the point $\bb +2 \be_1$ and travels along $\calL$ in the direction of the point $\bb - 2 \be_1$.
Now suppose that this random light ray reflects from $\bS^{d-1}$ according to the classical law of specular reflection, i.e., the angle of reflection is equal to the angle of incidence. Let $\bv\in \bS^{d-1}$ be the vector representing the direction of the reflected ray, i.e., the reflected light ray travels along a straight line of the form $\{\bw + a \bv, a\in \R\}$ for some vector $\bw \in \R^d$.

\begin{proposition}
The distribution of $\bv$ is uniform on $\bS^{d-1}$ if and only if $d=3$.
\end{proposition}

\begin{proof}
Let $\bn$ be the outer normal vector to the sphere $\bS^{d-1}$ at the point where the light ray hits the sphere. If $|\bb| = r_1$ and the angle between $\be_1 $ and $\bn$ is $\alpha_1$ then $r_1 = \sin \alpha_1$. 
Let $\Theta$ be the angle between
$\bv$ and $\be_1$. The specular law of reflection implies that
the angle between
$\bv$ and $\bn$ is $\alpha_1$ so $\Theta = 2 \alpha_1$. Hence, for a given $r\in(0,1)$, we have $|\bb| \leq r$ if and only if  $\Theta \leq 2 \alpha$, where $r = \sin \alpha$. Let $\beta = 2\alpha$ so that $r = \sin (\beta/2)$.
We obtain
\begin{align*}
\P(\Theta \leq \beta ) =
\P(|\bb| \leq r) = r^{d-1} = (\sin \alpha)^{d-1} =
(\sin (\beta/2))^{d-1}.
\end{align*}

Let $A_\beta$ be the spherical cap with the angle $\beta$, i.e., the set of points $x \in \bS^{d-1}$ such that the angle between the vector $\overrightarrow {0x}$ and $\be_1$ is smaller than or equal to $\beta$. Let $\mu$ be the uniform probability measure on $\bS^{d-1}$. It suffices to show that $\mu(A_\beta) = \P(\Theta \leq \beta )$ for all $\beta\in(0,\pi)$ if and only if $d=3$.

The following formulas for the area of $A_\beta$ and $\bS^{d-1}$ are taken from \cite{caparea}.
The area of $A_\beta$ is equal to $(2 \pi^{(d-1)/2} / \Gamma( (d-1)/2) ) \int_0^\beta \sin^{d-2} \gamma d \gamma$. The area of $\bS^{d-1}$ is $2\pi^{d/2}/ \Gamma(d/2)$. It follows that
\begin{align*}
\mu(A_\beta) = \frac{\Gamma(d/2)}{\sqrt{\pi}\Gamma( (d-1)/2) } \int_0^\beta \sin^{d-2} \gamma d \gamma.
\end{align*}
For $d=3$ and all $\beta\in(0,\pi)$,
\begin{align*}
\P(\Theta \leq \beta ) =
(\sin (\beta/2))^2
= \frac 12 (1-\cos \beta) 
= \frac{\Gamma(3/2)}{\sqrt{\pi}\Gamma( 1) } \int_0^\beta \sin \gamma d \gamma
= \mu(A_\beta),
\end{align*}
so the proposition is proved for $d=3$.

For all $d\geq 2$ and $\beta\in(0,\pi)$,
\begin{align*}
f(\beta)& := \frac {\prt}{\prt \beta} \P(\Theta \leq \beta )
= \frac {\prt}{\prt \beta} (\sin (\beta/2))^{d-1}
= \frac{d-1} 2 (\sin (\beta/2))^{d-2} \cos (\beta/2), \\
g(\beta)& := \frac {\prt}{\prt \beta}\mu(A_\beta)
= \frac{\Gamma(d/2)}{\sqrt{\pi}\Gamma( (d-1)/2) } \sin^{d-2}  \beta.
\end{align*}
This implies that
\begin{align*}
\frac{f(\pi/2)}{g(\pi/2)} \frac{g(\pi/4)}{f(\pi/4)} 
= 2^{(3/2)-d} \sec(\pi/8) (\sin(\pi/8))^{2-d}
= (2 \sin(\pi/8))^{3-d}.
\end{align*}
The last quantity is not equal to 1 for $d\ne 3$ so the functions $f$ and $g$ are not identically equal to each other. Hence,  for $d\ne 3$, it is not true that $\P(\Theta \leq \beta ) \equiv \mu(A_\beta)$.
\end{proof}

Since $d=3$ is the dimension of our physical space, this justifies the choice of the uniform direction of reflection in this paper. In other dimensions, we also assume that the direction of reflection is uniform, for several reasons. The first is mathematical convenience. Second, the assumption of the uniform angle of reflection allows us to use a Markov model for the process of locations of consecutive scattering events. Finally, we believe that due to mixing (in the probabilistic sense of the word),  our results would remain unchanged, in the qualitative sense, if we incorporated the true distribution of reflection in dimensions $d\ne 3$.
\end{proof}

\section*{Acknowledgments}

We are grateful to Zhenqing Chen, Tadeusz Kulczycki, Soumik Pal and Brent Werness for very helpful advice.

\bibliographystyle{elsart-num-sort}
\bibliography{gravity_bib1}

\end{document}